\documentclass[12pt,a4paper,reqno]{amsart}

\usepackage{rotating,array}
\usepackage{amsmath}
\usepackage{changepage}

\DeclareMathAlphabet{\mathpzc}{OT1}{pzc}{m}{it}

\title[Resonances for the cases $BC_2$/$C_2$]{Resonances for the Laplacian: \\the cases $BC_2$ and $C_2$ (except $\SO_0(p,2)$ with $p>2$ odd) }

\author{J. Hilgert}
\address{Department of Mathematics,
Paderborn University,
Warburger Str. 100,
D-33098 Paderborn,
Germany}
\email{hilgert@math.uni-paderborn.de}
\author{A. Pasquale}
\address{Universit\'e de Lorraine, Institut Elie Cartan de Lorraine, UMR CNRS 7502, Metz, F-57045, France}
\email{angela.pasquale@univ-lorraine.fr}
\author{T. Przebinda}
\address{Department of Mathematics, University of Oklahoma, Norman, OK 73019, USA}
\email{tprzebinda@ou.edu}

\newcommand{\leftFpar}{[\hskip-1.6pt[}
\newcommand{\rightFpar}{]\hskip-1.6pt]}
\newcommand{\bigleftFpar}{\big[\hskip-3.1pt\big[}
\newcommand{\bigrightFpar}{\big]\hskip-3pt\big]}

\newcommand{\eps}{\varepsilon}

\def\g{\mathfrak g}

\def\k{\mathfrak k}
\def\p{\mathfrak p}

\def\R{\mathbb{R}}
\def\C{\mathbb{C}}

\def\a{\mathfrak a}

\def\ss1{\mathfrak s_{\overline 1}}

\def\hs1{\mathfrak h_{\overline 1}}

\def\Dg{\mathrm{D}}
\def\G{\mathrm{G}}

\def\K{\mathrm{K}}

\def\M{\mathrm{M}}

\def\Sg{\mathrm{S}}

\def\Spin{\mathrm{Spin}}

\def\Bbb{\mathbb}

\def\SL{\mathrm{SL}}
\def\SO{\mathrm{SO}}
\def\SU{\mathrm{SU}}
\def\Sp{\mathrm{Sp}}

\def\Ug{\mathrm{U}}

\def\Eg{\mathrm{E}}

\def\id{\mathrm{id}}
\def\B{\mathrm{B}}
\def\rl{\mathrm{l}}
\def\mrm{\mathrm{m}}
\def\rs{\mathrm{s}}

\def \wt{\widetilde}

\def\bi{\mathbf{i}}

\def\W{\mathsf{W}}

\def\X{\mathsf{X}}

\def\Zb{\mathbb {Z}}

\def\Re{\mathop{\hbox{\rm Re}}\nolimits}
\def\Im{\mathop{\hbox{\rm Im}}\nolimits}

\def\lim{\mathop{\hbox{\rm lim}}\nolimits}

\newcommand\inner[2]{\langle #1,#2\rangle}

\def\Res{\mathop{\hbox{\rm Res}}\limits}

\newcommand{\cHC}{c_{\rm{HC}}}

\newcommand{\cz}{{\mathop{\mathrm c}}}
\newcommand{\sz}{{\mathop{\mathrm s}}}

\def\fonttitre{\textsf}
\newcounter{thh}

\newtheorem{thm}[thh]{\fonttitre{Theorem}}

\newtheorem{pro}[thh]{\fonttitre{Proposition}}
\newtheorem*{pro*}{\fonttitre{Proposition}}
\newtheorem{cor}[thh]{\fonttitre{Corollary}}
\newtheorem*{coro*}{\fonttitre{Corollary}}
\newtheorem{lem}[thh]{\fonttitre{Lemma}}

\theoremstyle{definition} \newtheorem{rem}{\fonttitre{Remark}}

\newtheorem*{defi*}{\fonttitre{D??finition}}

\newtheorem*{nota*}{\fonttitre{Notation}}
\def\muet{ \ifthenelse{\equal{a}{b}}}
\def\nn{\nonumber}

\begin{document}


%

\maketitle
\tableofcontents

\section*{Introduction}
The study of resonances has started in quantum mechanics, where they are linked to the metastable states of a system. Mathematically, the resonances appear as poles of the meromorphic continuation of the resolvent $(H-z)^{-1}$ of a Hamiltonian $H$ acting on a space of functions $\mathcal{F}$ on which $H$ is not selfadjoint. In the last thirty years, several articles have considered the case where $H$ is the Laplacian of a Riemannian symmetric space of the noncompact type $\X$ and $\mathcal{F}$ is the space $C_c^\infty(\X)$ of smooth compactly supported functions on $\X$. The basic problems are the existence, location, counting estimates and geometric interpretation of the resonances. All these problems are nowadays well understood when $\X$ is of real rank one, such as the real hyperbolic spaces. 
The situation is completely different for Riemannian symmetric spaces of higher rank.  The
pioneering articles proving the analytic continuation of the resolvent of the Laplacian operator across its continuous spectrum are \cite{MV05} and \cite{Str05}. However, in these articles, the domains where the continuation was obtained is not sufficiently large to cover the region where the resonances could possibly be found. Indeed, the existence of resonances is linked to the singularities of the Plancherel measure on $\X$. The basic question, whether resonances exist or not for general Riemannian symmetric spaces for which the Plancherel measure is singular, is still open. If the general picture is still unknown, some complete examples in rank 2 have been treated recently: $\SL(3,\R)/\SO(3)$ in \cite{HPP14} and 
the direct products $\X_1\times \X_2$ of two rank-one Riemannian symmetric spaces of the noncompact type in \cite{HPP15}.

The present paper is a natural continuation of \cite{HPP15} and deals with the cases of 
Riemannian symmetric spaces $\X=\G/\K$ of real rank two and restricted root system $BC_2$ or $C_2$ except the case when $\G=\SO_0(p,2)$ with $p>2$ odd. The reason is that for all the spaces $\X$ considered here the analysis of the meromorphic continuation of the resolvent of the Laplacian can be deduced from the same problem on a direct product $\X_1\times \X_1$ of a Riemannian symmetric space of rank one not isomorphic to the real hyperbolic space. 

We prove that for all the spaces $\X$ we consider, the resolvent of the Laplacian of $\X$ can be lifted to a meromorphic function on a Riemann surface which is a branched covering of $\mathbb C$. Its poles, that is the resonances of the Laplacian, are explicitly located on this Riemann surface. If $z_0$ is a resonance of the Laplacian, then the (resolvent) residue operator at $z_0$ is the linear operator 
\begin{equation}
\label{residueop1}
{\Res}_{z_0} \wt R: C^\infty_c(\X)\to C^\infty(\X)
\end{equation}
defined by 
\begin{equation}
\label{residueop2}
\big({\Res}_{z_0} \wt R f\big)(y)={\Res}_{z=z_0} [R(z)f](y) \qquad (f\in C_c^\infty(\X), \, y\in \X)\,.
\end{equation}
Since the meromorphic extension takes place on a Riemann surface, the right-hand side of 
(\ref{residueop2}) is computed with respect to some coordinate charts and hence determined up to constant multiples. However, the image ${\Res}_{z_0}\wt R\big(C^\infty_c(\X)\big)$ is a well-defined subspace of $C^\infty(\X)$. 
Its dimension is the rank of the residue operator at $z_0$. We prove that 
${\Res}_{z_0}\wt R$ acts on $C^\infty_c(\X)$ as a convolution by a finite linear combination of 
spherical functions of $\X$ and is of finite rank. More precisely, write $\X=\G/\K$ for a 
connected noncompact real semisimple Lie group with finite center $\G$ with maximal compact subgroup $\K$. Then the space ${\Res}_{z_0}\wt R\big(C^\infty_c(\X)\big)$ is a $\G$-module which is a finite direct sum of finite-dimensional irreducible spherical representations of $\G$. The trivial representation of $\G$
occurs for the residue operator at the first singularity, associated with the bottom of the spectrum of the Laplacian. 

\subsubsection*{Acknowledgements} The second author would like to thank the University of Oklahoma as well as the organizers of the XXXV Workshop on Geometric Methods in Physics, Bialowieza, for their hospitality and financial support.
The third author gratefully acknowledges partial support from the NSA grant H98230-13-1-0205. 

\medskip

\section{Preliminaries}

\subsection{General notation}
We use the standard notation $\mathbb Z$, $\R$,  $\R^+$, $\C$ and $\C^\times$ for the integers, the reals, the positive reals, the complex numbers and the non-zero complex numbers, respectively. For $a\in \Zb$, the symbol $\mathbb Z_{\geq a}$ denotes the set of integers $\geq a$. We write $\leftFpar a,b \rightFpar=[a,b]\cap \Zb$ for the discrete interval of integers in $[a,b]$.
The interior of an interval $I\subseteq \R$ (with respect to the usual topology on the real line) will be indicated by $I^\circ$.
The upper half-plane in $\C$ is $\C^+=\{z \in \C:\Im z>0\}$; the lower half-plane $-\C^+$ is denoted $\C^-$.
If $\X$ is a manifold, then $C^\infty(\X)$ and $C^\infty_c(\X)$ respectively denote the space of smooth functions and the space of smooth compactly supported functions on $\X$.

\subsection{Noncompact irreducible Riemannian symmetric spaces of type $BC_2$ or $C_2$}
\label{subsection:symmspaces}
Let $\X=\G/\K$ be an irreducible Riemannian symmetric space of the noncompact type and (real) rank $2$. Hence $\G$ is a connected noncompact semisimple real Lie group with finite center and $\K$ is a maximal compact subgroup of $\G$. We can suppose that $\G$ is simple and admits a faithful linear representation.  Let $\g$ and $\k$ be respectively the Lie algebras of $\G$ and $\K$, and let $\g=\k\oplus \p$ be the corresponding Cartan decomposition. 
Let us fix a maximal abelian subspace $\a$ of $\p$.  The (real) rank $2$ condition means that 
$\a$ is a 2-dimensional real vector space. We denote by $\a^*$ the dual space of $\a$ and by
$\a_\C^*$ the complexification of $\a^*$. The Killing form of $\g$ restricts to an inner product on $\a$. We extend it to $\a^*$ by duality. The $\C$-bilinear extension of $\inner{\cdot}{\cdot}$ to $\a_\C^*$ will be indicated by the same symbol. 

Let $\Sigma$ be the root systems of $(\g,\a)$.
In the following, we suppose that $\Sigma$ is either of type $BC_2$ or of type $C_2=B_2$.
The set $\Sigma^+$ of positive restricted roots is the form
$\Sigma^+=\Sigma^+_\rl \sqcup \Sigma^+_\mrm \sqcup \Sigma^+_\rs$, where
\begin{eqnarray}
\label{eq:roots}
\Sigma^+_\rl=\{\beta_1, \beta_2\}\,,  \notag\, \quad
\Sigma^+_\mrm=\big\{\frac{\beta_2\pm \beta_1}{2}\}\,,\quad
\Sigma^+_\rs=\big\{\frac{\beta_1}{2}, \frac{\beta_2}{2}\big\} 
\end{eqnarray}
with $\Sigma^+_\rs=\emptyset$ in the case $C_2=B_2$.
The two elements of $\Sigma^+_\rl$ form an orthogonal basis of $\a^*$ and have same norm $b$. The elements of $\Sigma^+_\mrm$ and $\Sigma^+_\rs$ have therefore norm $\frac{\sqrt{2}}{2} b$ and $\frac{b}{2}$, respectively. We define
$\mathfrak{a}^*_+=\{\lambda\in \mathfrak{a}^*: \text{$\inner{\lambda}{\beta}>0$ for all $\beta\in\Sigma^+$}\}$.

The system of positive unmultipliable roots is $\Sigma_*^+=\Sigma^+_\rl \sqcup \Sigma^+_\mrm$.
The set $\Sigma_*$ of unmultipliable roots is a root system. A basis of positive simple roots
for $\Sigma_*$ is $\big\{\beta_1\,,\frac{\beta_2- \beta_1}{2}\}$.

The Weyl group $\W$ of $\Sigma$ acts on the roots by permutations and sign changes. For  $a\in \{\rl,\mrm,\rs\}$ set $\Sigma_a=\Sigma^+_a \sqcup (-\Sigma^+_a)$.  Then each $\Sigma_a$ is a Weyl group orbit in $\Sigma$.
The  root multiplicities are therefore triples $m=(m_\rl,m_\mrm,m_\rs)$ so that $m_a$ is the (constant) value of $m$ on $\Sigma_a$ for $a\in \{\rl,\mrm,\rs\}$. By classification, if $\X=\G/\K$ is Hermitian, then $m_\rl=1$. We adopt the convention that $m_\rs=0$ means that $\Sigma^+_\rs=\emptyset$, i.e. $\Sigma$ is of type $C_2$. In this case, if $\X$ is Hermitian, then  $\X$ is said to be of tube type.

The half sum of positive roots, counted with their multiplicities, is indicated by $\rho$. Hence
\begin{equation}
\label{eq:rho}
 2\rho=\sum_{\alpha \in \Sigma^+} m_\alpha \alpha= \Big(m_\rl+ \frac{m_\rs}{2}\Big) \beta_1 + \Big(m_\rl+m_\mrm+\frac{m_\rs}{2}\Big) \beta_2\,.
\end{equation}

Table 1 contains the rank-two irreducible Riemannian symmetric spaces $\G/\K$ with root systems of type $BC_2$, their root systems, the multiplicities $m=(m_\rl,m_\mrm,m_\rs)$, and the value of $\rho$.

\smallskip
\begin{table}[!ht]
\begin{adjustwidth}{-1.5cm}{}
\setlength{\extrarowheight}{.2em}
\begin{tabular}{|c||c|c|c|c|c|}
\hline
Type &AIII & BDI & CII & DIII  & EIII\\[.2em]
\hline
$\G$ & $\SU(p,2)$ $(p>2)$ &$\SO_0(p,2)\;  (p>2)$  & $\Sp(p,2)\;  (p\geq 2)$
& $\SO^*(10)$ & $E_{6(-14)}$
 \\[.2em]
\hline
$\K$ & $\Sg(\Ug(p)\times \Ug(2))$ &$\SO(p)\times\SO(2)$ & $\Sp(p)\times\Sp(2)$  &
$\Ug(5)$  & $\Spin(10)\times\Ug(1)$
 \\[.2em]
\hline
\scriptsize{Hermitian}  & yes & yes  & no  &  yes  & yes
\\[.2em]
\hline
$\Sigma$ & $BC_2$  & $C_2$  & \begin{tabular}{l}
$p=2$: $C_2$ \\
$p>2$: $BC_2$
\end{tabular}  & $BC_2$ & $BC_2$
\\[.2em]
\hline
{\footnotesize $m=(m_\rl,m_\mrm,m_\rs)$} &  {\footnotesize $(1,2,2(p-2))$} & {\footnotesize $(1,p-2,0)$} &
{\footnotesize 
$(3,4,4(p-2))$}
& {\footnotesize $(1,4,4)$}  & {\footnotesize $(1,6,8)$} \\[.2em]
\hline
 $2\rho$ & {\footnotesize $(p-1)\beta_1+(p+1)\beta_2$} & {\footnotesize $\beta_1+(p-1)\beta_2$} & {\footnotesize $5\beta_1+(5+2(p-2))\beta_2$} & {\footnotesize $3\beta_1+7\beta_2$}  & {\footnotesize $5\beta_1+8\beta_2$}
  \\[.2em]
\hline
\end{tabular}
\medskip
\caption{Rank-two irreducible symmetric spaces with root systems $BC_2$ or $C_2$}
\end{adjustwidth}
\end{table}

Notice that we are using special low rank isomorphisms (see e.g. \cite[Ch. X, \S 6, no.4]{He1}), which allow us to omit some cases:
\begin{eqnarray}
\SU(2,2)/\Sg(\Ug(2)\times \Ug(2)) &\cong & \SO_0(4,2)/ (\SO(4)\times \SO(2))\,,\\
\Sp(2,\R)/\Ug(2) &\cong &\SO_0(3,2)/ (\SO(3)\times \SO(2))\,,\\
\SO^*(8)/\Ug(4) &\cong &\SO_0(6,2)/ (\SO(6)\times \SO(2))\,.
\end{eqnarray}
Observe also that $\SO_0(2,2)/(\SO(2)\times \SO(2)) \cong \SL(2,\R)\times \SL(2,\R)$ is not in the list because not irreducible.

\begin{rem}
Up to isomorphisms, there are four additional irreducible Riemannian symmetric spaces of rank two:
\begin{enumerate}
\item
$\SL(3,\R)/\SO(3)$ (type AI, with root system of type $A_2$ and one root multiplicity $m=1$; see \cite{HPP14}),
\item
$\SU^*(6)/\Sp(3)$ (type AII, with root system of type $A_2$ and one even root multiplicity $m=4$; see \cite{Str05}),
\item
$E_{6(-26)}/F_4$ (type EIV, with root system of type $A_2$ and one even root multiplicity $m=8$; see \cite{Str05}),
\item
$G_{2(-14)}/(\SU(2)\times \SU(2))$ (type G, with root system of type $G_2$ and one root multiplicity
$m=1$).
\end{enumerate}
\end{rem}

\subsection{The Plancherel density of $\G/\K$}
For $\lambda \in \mathfrak a_\C^*$ and $\beta \in \Sigma$ we shall employ the notation
\begin{equation}\label{eq:la}
\lambda_\beta=\frac{\inner{\lambda}{\beta}}{\inner{\beta}{\beta}}\,.
\end{equation}
Observe that
\begin{eqnarray}
\label{eq:lambdabetahalf}
\lambda_{\beta/2}&=&2\lambda_{\beta}\,,\\
\label{eq:lambdabeta12}
\lambda_{(\beta_2\pm \beta_1)/2}&=&
2  \frac{\inner{\lambda}{\beta_2}\pm \inner{\lambda}{\beta_1}}
{\inner{\beta_2}{\beta_2}+ \inner{\beta_1}{\beta_1}}=\lambda_{\beta_2}\pm \lambda_{\beta_1}\,.
\end{eqnarray}
For $\beta \in \Sigma_*$, we set
\begin{equation}
\label{eq:wtrhob}
\wt\rho_\beta=\frac{1}{2} \left( m_{\beta}+ \frac{m_{\beta/2}}{2}\right)\,,
\end{equation}
where $m_\beta$ denotes the multiplicity of the root $\beta$, and
\begin{equation}
\label{eq:cbeta}
c_\beta(\lambda)= \frac{2^{-2\lambda_\beta}\Gamma(2\lambda_\beta)}{\Gamma\big(\lambda_\beta+\frac{m_{\beta/2}}{4}+\frac{1}{2}\big) \Gamma\big(\lambda_\beta+\wt\rho_\beta\big)}\,.
\end{equation}
Observe that $\wt\rho_\beta=\rho_\beta=\frac{\langle \rho,\beta\rangle}{\langle\beta,\beta\rangle}$ if $\beta$ is a simple root in $\Sigma_*$ (but not in general).
In particular, $\wt\rho_{\beta_1}=\rho_{\beta_1}$.

Harish-Chandra's $c$-function $\cHC$ (written in terms of unmultipliable
instead of indivisible roots) is defined by 
\begin{equation}
 \label{eq:c}
\cHC(\lambda)=c_0 \; \prod_{\beta\in\Sigma_*^+} c_\beta(\lambda)\,.
\end{equation}
where $c_0$ is a normalizing constants so that $\cHC(\rho)=1$.

In the following we always adopt the convention that empty products are equal to $1$.
As a consequence of the properties of the gamma function, we have the following explicit expression.

\begin{lem}
\label{lemma:Planch-density}
The Plancherel density is given by the formula
\begin{equation}
\label{eq:Planch-density}
[\cHC(\lambda)\cHC(-\lambda)]^{-1}=C \Pi(\lambda) P(\lambda) Q(\lambda)\,,
\end{equation}
where
\begin{eqnarray}
\label{eq:Pi}
\Pi(\lambda)&=&\prod_{\beta \in \Sigma_*^+} \lambda_\beta\,,\\
\label{eq:P}
P(\lambda)&=&
\prod_{\beta \in \Sigma_*^+} \Big( \prod_{k=0}^{(m_{\beta/2})/2-1} \big[ \lambda_\beta -\big( \tfrac{m_{\beta/2}}{4}-\tfrac{1}{2} \big)+k\big]
\prod_{k=0}^{2 \wt \rho_\beta-2} [ \lambda_\beta -(\wt \rho_\beta-1)+k]\Big)\,, \\
\label{eq:Q}
Q(\lambda)&=&\prod_{\stackrel{\beta\in \Sigma_*^+}{\textup{$m_\beta$ odd}}}
 \cot(\pi(\lambda_\beta - \wt\rho_\beta))\,.
\end{eqnarray}
and $C$ is a constant.
Consequently, the singularities of the Plancherel density $[\cHC(\lambda)\cHC(-\lambda)]^{-1}$  are at most simple poles located along the hyperplanes of the equation
\begin{equation*}
\pm \lambda_\beta=\wt\rho_\beta +k
\end{equation*}
where $\beta \in \Sigma_*^+$ has odd multiplicity $m_\beta$, and  $k \in \Zb_{\geq 0}$.
\end{lem}
\begin{proof}
The singularities of $[\cHC(\lambda)\cHC(-\lambda)]^{-1}$ are those of $\cot(\pi(\lambda_\beta-\wt\rho_\beta))$,
for $\beta \in \Sigma_*^+$ with $m_ \beta$ odd, which are not killed by zeros of the polynomial
$\Pi(\lambda) P(\lambda)$.
\end{proof}

The following corollary will allow us to establish a region of holomorphic extension of the
resolvent.

\begin{cor}
Set
\begin{equation}
\label{eq:L}
L=\min\{\wt\rho_\beta |\beta|: \text{$\beta \in \Sigma_*^+$, $m_\beta$ odd} \}\,.
\end{equation}
Then, for every fixed $\omega \in \mathfrak{a}^*$ with $|\omega|=1$,  the function
$r \mapsto [\cHC(r\omega)\cHC(-r\omega)]^{-1}$ is holomorphic on $\C \setminus  \big(]-\infty,-L] \cup [L,+\infty[ \big)$.
\end{cor}

The values of $\wt\rho_\beta$ for the roots in $\Sigma_*^+$, as well as the value of
$L$, are given in Table 2. 
Recall that 
$b=\langle\beta_1,\beta_1\rangle=\langle\beta_2,\beta_2\rangle$.
\smallskip
\begin{table}[!ht]
\setlength{\extrarowheight}{.2em}
\begin{adjustwidth}{-1.5cm}{}
\begin{tabular}{|c||c|c|c|c|c|}
\hline
$\G$ & $\SU(p,2)$ $(p>2)$ &$\SO_0(p,2)\;  (p>2)$  & $\Sp(p,2)\;  (p\geq 2)$
& $\SO^*(10)$ & $E_{6(-14)}$
 \\[.2em]
\hline\hline
$\wt \rho_{\beta_j}=\frac{1}{2}\big(m_\rl+\frac{m_\rs}{2}\big)$  &
$(p-1)/2$ & $1/2$ & $p-1/2$ & $3/2$ & $5/2$ \\[.2em]
\hline
$\wt \rho_{(\beta_2\pm \beta_1)/2}=\frac{m_\mrm}{2}$ &
1 & $(p-2)/2$ & $2$ & $2$ & $3$ \\[.2em]
\hline
\footnotesize{$\Sigma^+_{*,{\rm odd}}=\{\beta\in \Sigma_*^+:\text{$m_\beta$ odd}\}$ } &
$\{\beta_1,\beta_2\}$ &
\begin{tabular}{l}
\footnotesize{$p$ even:  $\{\beta_1,\beta_2\}$} \\
\footnotesize{$p$ odd: $\{\beta_1,\beta_2, \frac{\beta_2\pm \beta_1}{2}\}$}
\end{tabular}
& $\{\beta_1,\beta_2\}$ & $\{\beta_1,\beta_2\}$ & $\{\beta_1,\beta_2\}$ \\[.2em]
\hline
{\footnotesize $L=\min\{\wt\rho_\beta |\beta|:
\text{$\beta \in \Sigma^+_{*,{\rm odd}}$}\} $} & $\frac{\sqrt{p-1}}{2}b$ &
\begin{tabular}{l}
\footnotesize{$p=3$: $\frac{\sqrt{2}}{4}b$} \\
\footnotesize{$p> 3$: $\frac{b}{2}$}
\end{tabular}   & $\big(\frac{3}{2}+2(p-2)\big) b$ & $\frac{3}{2}b$ & $\frac{5}{2}b$ \\[.2em]
\hline
\end{tabular}
\medskip
\caption{The values of $\wt \rho_\beta$ for $\beta \in \Sigma_*^+$ and of $L$}\label{table tilde rho}
\end{adjustwidth}
\end{table}

A computation using the values in the tables together with \cite[\S 2]{HPP15} yields the following corollary.

\begin{cor}\label{Cor:Plancherel density}
If $\G\neq \SO_0(p,2)$ with $p$ odd, then $\{\beta\in \Sigma_*^+: \text{$m_\beta$ is odd}\}$
is equal to $\{\beta_1,\beta_2\}$. Hence
\begin{equation}
\label{eq:cHCandproducts}
[\cHC(\lambda)\cHC(-\lambda)]^{-1}=
\Pi_0(\lambda)P_0(\lambda)[\cHC^\times(\lambda)\cHC^\times(-\lambda)]^{-1}\,,
\end{equation}
where
\begin{eqnarray}
\label{eq:Pi0}
\Pi_0(\lambda)&=&\lambda_{(\beta_2-\beta_1)/2} \lambda_{(\beta_2+\beta_1)/2}
=\lambda_{\beta_2}^2-\lambda_{\beta_1}^2\,,\\
\label{eq:P0}
P_0(\lambda)&=&
\prod_{\beta=(\beta_2\pm \beta_1)/2}
\prod_{k=0}^{2 \wt \rho_\beta-2} [ \lambda_\beta -(\wt \rho_\beta-1)+k]\,,
\end{eqnarray}
and
$[\cHC^\times(\lambda)\cHC^\times(-\lambda)]^{-1}$ is the Plancherel density of the product
$\X_1\times \X_1$ of two isomorphic rank-one Riemannian symmetric spaces with root systems of type $BC_1$ (or $A_1$) and multiplicities $(m_{\beta_j},m_{\beta_j/2})=(m_\rl,m_\rs)$.

If $\G=\SO_0(p,2)$ with $p\geq 3$ odd, then $\Sigma^+_*=\Sigma^+$ and
\begin{eqnarray*}
\Pi(\lambda)&=&\lambda_{\beta_1} \lambda_{\beta_2} (\lambda_{\beta_2}^2-\lambda_{\beta_1}^2)\\
P(\lambda)&=&\prod_{k=0}^{p-2} \Big( \lambda_{(\beta_2-\beta_1)/2}-\big(\frac{p-2}{2}-1\big)+k\Big)\Big( \lambda_{(\beta_2+\beta_1)/2}-\big(\frac{p-2}{2}-1\big)+k\Big)\\
&=&\prod_{k=0}^{p-2} \Big( \lambda_{\beta_2}-\lambda_{\beta_1}-\big(\frac{p-2}{2}-1\big)+k\Big)\Big(  \lambda_{\beta_2}+\lambda_{\beta_1}-\big(\frac{p-2}{2}-1\big)+k\Big)\\
Q(\lambda)&=&\cot\big(\pi \big( \lambda_{\beta_1}-\tfrac{1}{2}\big)\big)
\cot\big(\pi \big( \lambda_{\beta_2}-\tfrac{1}{2}\big)\big)
\cot\big(\pi \big(  \lambda_{\beta_2}-\lambda_{\beta_1}-\tfrac{p}{2}+1\big)\big)\\
&&\hskip 7truecm \times \cot\big(\pi \big( \lambda_{\beta_2}+\lambda_{\beta_1}-\tfrac{p}{2}+1\big)\big)
\end{eqnarray*}
\end{cor}

\subsection{The resolvent of $\Delta$}

Endow the Euclidean space $\mathfrak{a}^*$ with the Lebesgue measure normalized so that the unit cube has volume $1$. On the Furstenberg boundary $\B=\K/\M$ of $\X$, where $\M$ is the centralizer of $\mathfrak{a}$ in $\K$, we consider the $\K$-invariant measure $db$ normalized so that the volume of $\B$ is equal to $1$.
Let $\X$  be equipped with its (suitably normalized) natural $\G$-invariant Riemannian measure
and let $\Delta$ denote the corresponding (positive) Laplacian. 
As in the cases treated in \cite{HPP14} and \cite{HPP15}, it will be convenient to identify $\mathfrak{a}^*$ with $\C$ as vector spaces over $\R$. More precisely, we want to view $\mathfrak{a}_1^*$ and $\mathfrak{a}_2^*$ as the real and the purely imaginary axes, respectively. To distinguish the resulting complex structure in $\mathfrak{a}^*$ from the natural complex structure of $\mathfrak{a}^*_\C$, we shall indicate the complex units in $\mathfrak{a}^*\equiv \C$ and $\mathfrak{a}_\C^*$ by $i$ and $\bi$, respectively.
So $\mathfrak{a}^*\equiv \C=\R +i \R$, whereas $\mathfrak{a}_\C^*=\mathfrak{a}^*+\bi\mathfrak{a}^*$. For $r,s\in \R$ and $\lambda,\nu\in \mathfrak{a}^*$ we have $(r+is)(\lambda+\bi\nu)=(r\lambda-s\nu)+\bi(r\nu+s\lambda)\in \mathfrak{a}^*_\C$.

By the Plancherel Theorem \cite[Ch. III, \S 1, no. 2]{He3}, the Helgason-Fourier transform
$\mathcal F$ is a unitary equivalence of $\Delta$ acting on $L^2(\X)$ with the
multiplication operator $M$ on $L^2(\mathfrak{a}^*_+\times B,[\cHC(\bi \lambda)\cHC(-\bi \lambda)]^{-1} \; d\lambda\, db)$ given by
\begin{equation} \label{eq:Mmult}
MF(\lambda,b)=\Gamma(\Delta)(\bi\lambda)F(\lambda,b)=(\inner{\rho}{\rho}+\inner{\lambda}{\lambda})F(\lambda,b)  \qquad ((\lambda,b)\in \mathfrak{a}^*\times B)\,.
\end{equation}
It follows, in particular, that the spectrum of $\Delta$ is the half-line
$[\rho_\X^2, +\infty[$, where $\rho_\X^2=\inner{\rho}{\rho}$.
By the Paley-Wiener theorem for $\mathcal{F}$, see e.g. \cite[Ch. III, \S 5]{He3}, for every
$u \in \C \setminus [\rho_\X^2, +\infty[$ the resolvent of $\Delta$ at $u$ maps $C^\infty_c(\X)$ into $C^\infty(\X)$.

Recall that for sufficiently regular functions $f_1, f_2:\X \to \C$, the convolution $f_1\times f_2$ is the function on $\X$ defined by $(f_1 \times f_2) \circ \pi= (f_1 \circ \pi) * (f_2 \circ \pi)$. Here $\pi:\G\to \X=\G/\K$ is the natural projection and $*$ denotes the convolution product of functions on $\G$.

The Plancherel formula yields the following explicit expression for the image of $f \in C^\infty_c(\X)$ under the resolvent operator $R(z)=(\Delta-\rho_\X^2-z^2)^{-1}$ of the shifted Laplacian $\Delta - \rho_\X^2$:
\begin{equation}\label{eq:resolvent}
[R(z)f](y)=\int_{\mathfrak{a}^*} \frac{1}{\inner{\lambda}{\lambda}-z^2}\; (f \times \varphi_{\bi\lambda})(y) \; \frac{d\lambda}{\cHC(\bi\lambda)\cHC(-\bi\lambda)}
\qquad (z \in \C^+\,, y\in \X)\,.
\end{equation}
See \cite[formula (14)]{HP09}. Here and in the following, resolvent equalities as \eqref{eq:resolvent} are given up to non-zero constant multiples.

\section{Meromorphic extension in the case $\G\neq \SO_0(2,p)$, $p>2$ odd}

\subsection{The resolvent kernel in polar coordinates}
We write $\mathfrak{a}_1^*=\R \beta_1$ and $\mathfrak{a}_2^*=\R \beta_2$, so that $\mathfrak{a}^*=\mathfrak{a}_1^*\oplus\mathfrak{a}^*_2$ and
$\lambda=\lambda_1+\lambda_2=x_1\beta_1+x_2\beta_2\in \mathfrak{a}^*$. 
Introduce the coordinates
\begin{equation}\label{1.1}
\R^2\ni (x_1,x_2)\to x_1\beta_1+x_2\beta_2\in \a_1^*\oplus\a_2^*=\a^*.
\end{equation}
Hence $x_j=\lambda_{\beta_j}$ if $\lambda=x_1\beta_1+x_2\beta_2$.

In view of Table~\ref{table tilde rho}, the functions $\Pi_0$ and $P_0$ from \eqref{eq:Pi0} and \eqref{eq:P0} can be rewritten in these coordinates, as
\begin{eqnarray}
\label{eq:Pi0'}
\Pi_0(\lambda)&=&\Pi_0(x_1\beta_1+x_2\beta_2)\ =\ x_2^2-x_1^2\,,\\
\label{eq:P0'}
P_0(\lambda)&=&\nonumber P_0(x_1\beta_1+x_2\beta_2)
\ =\ \prod_{\beta=(\beta_2\pm \beta_1)/2} \prod_{k=0}^{m_\mrm-2} \big[ \lambda_\beta -\big(\frac{m_\mrm}{2}-1\big)+k\big]\\
&=&\nonumber \prod_{k=1}^{m_\mrm-1} \big[ (x_2+x_1)-\frac{m_\mrm}{2}+k\big]\,\big[ (x_2-x_1)-\frac{m_\mrm}{2}+k\big]\\
&=&\prod_{k=1}^{m_\mrm-1} \big[ (x_2-\frac{m_\mrm}{2}+k)^2-x_1^2\big]
\end{eqnarray}
since
$$
(\beta_j)_{(\beta_2\pm\beta_1)/2}=(\beta_j)_{\beta_2}\pm (\beta_j)_{\beta_1}=
\begin{cases}
 \pm 1 &\text{for } j=1,\\
 1 &\text{for } j=2.\\
\end{cases}
$$
We write
\begin{equation}\label{eq:vartheta0}
\vartheta_0(x_1,x_2)= \Pi_0(\lambda)P_0(\lambda) = (x_2^2-x_1^2)\prod_{k=1}^{m_\mrm-1} \big[ (x_2-\frac{m_\mrm}{2}+k)^2-x_1^2\big]
\end{equation}
Further we write
\begin{equation}\label{eq:Rang 1 c-function}
 [\cHC^{\X_1}(\bi \lambda)\cHC^{\X_1}(-\bi \lambda)]^{-1} = C_1 \Pi_1(\bi \lambda) P_1(\bi \lambda) Q_1(\bi \lambda)
\end{equation}
for the Plancherel density of the space $\X_1$ in Corollary~\ref{Cor:Plancherel density}, so that
\begin{equation}\label{eq:Rang 1 c-function}
 [\cHC^\times(\bi\lambda)\cHC^\times(-\bi\lambda)]^{-1} = C_1^2 \Pi_1(i x_1)\Pi_1(i x_2) P_1(i x_1)P_1(i x_2) Q_1(i x_1)Q_1(i x_2).
\end{equation}
See \cite[\S 1 and \S 2]{HPP15}. 
Using \eqref{eq:cHCandproducts} and omitting non-zero constant multiples, we can therefore rewrite \eqref{eq:resolvent} as
\begin{multline*}
R(z)f(y)
=\int_{\a^*}\frac{1}{\langle \lambda,\lambda\rangle-z^2}(f\times \varphi_{\bi\lambda})(y)\frac{1}{\cHC(\bi\lambda)\cHC(-\bi\lambda)}\,d\lambda\\
=\int_{\R^2}\frac{(f\times \varphi_{\bi x_1\beta_1+\bi x_2\beta_2})(y)}{x_1^2b^2+x_2^2b^2-z^2} \vartheta_0(ix_1,ix_2)x_1x_2P_1(ix_1)P_1(ix_2)Q_1(ix_1)Q_1(ix_2)\,dx_1\,dx_2\,.
\end{multline*}
Introduce the polar coordinates
\[
x_1=\frac{r}{b}\cos \theta,\ \ \ x_2=\frac{r}{b}\sin \theta \qquad (0<r\,,\ 0\leq \theta<2\pi)
\]
on $\R^2$ and set
\begin{equation}
\label{eq:pq}
 p_1(x)=P_1\big(i\tfrac{x}{b}\big)\qquad \text{and} \qquad q_1(x)=Q_1\big(i\tfrac{x}{b}\big).
\end{equation}
In these terms (up to a non-zero constant multiple)
\[
R(z)f(y)=\int_0^\infty\frac{1}{r^2-z^2}F(r)\,r\,dr,
\]
where
\begin{eqnarray}\label{2.1}
F(r)&=&\int_0^{2\pi}(f\times \varphi_{\bi\frac{r}{b}\cos \theta\,\beta_1+\bi\frac{r}{b}\sin \theta\,\beta_2})(y)\,\vartheta_{0,{\rm pol}}(r,\theta)r^2\cos\theta \sin\theta\\
&&\times \; p_1(r\cos \theta) q_1(r\cos \theta) p_1(r\sin \theta) q_1(r\sin \theta)\,d\theta\,,\nn
\end{eqnarray}
where
$$\vartheta_{0,{\rm pol}}(r,\theta)= \vartheta_0(i\,x_1,i\,x_2) = -\frac{r^2}{b^2}(\sin^2 \theta -\cos^2 \theta) \prod_{k=1}^{m_\mrm-1} \big[ \big(\frac{r}{b}i\sin \theta-\frac{m_\mrm}{2}+k\big)^2+\frac{r^2}{b^2}\cos^2 \theta\big].$$
Here and in the following, we omit from the notation the dependence of $F$ on the function $f\in C^\infty_c(\X)$ and on $y\in \X$.

Recall the functions
\begin{equation}
\label{eq:c-s}
\cz(w)=\frac{w+w^{-1}}{2}\,,\qquad  \sz(w)=\frac{w-w^{-1}}{2}=i\cz(-iw) \qquad (w\in\C^\times)
\end{equation}
from \cite[(20)]{HPP15} and notice that
\[
\cos\theta=\cz(e^{i\theta})\,, \qquad \sin\theta=\frac{\sz(e^{i\theta})}{i}=\cz(-ie^{i\theta})\,,\qquad d\theta=\frac{d e^{i\theta}}{ie^{i\theta}}.
\]
For $z \in \C$ and $w \in \C^\times$ define
\begin{eqnarray}
\label{eq:psiz}
\psi_z(w)&=&(f\times \varphi_{\bi\frac{z}{b}\cz(w)\,\beta_1+\bi\frac{z}{b}\cz(-iw)\,\beta_2})(y)\\
\label{eq:phiz}
\phi_z(w)&=&-z^2\cz(w) \frac{\sz(w)}{w} p_1\big(z\cz(w)\big) q_1\big(z\cz(w)\big) p_1\big(z\cz(-iw)\big) q_1\big(z \cz(-iw)\big).
\end{eqnarray}
as in \cite[(32) and (33)]{HPP15} together with
\begin{equation}\label{eq:vathetaz}
  \vartheta_z(w)=\frac{z^2}{b^2}\big(\cz(w)^2-\cz(-iw)^2\big)\prod_{k=1}^{m_\mrm-1} \big[ \big(\frac{z}{b}\sz(w)-\frac{m_\mrm}{2}+k\big)^2+\frac{z^2}{b^2}\cz(w)^2\big]\,,
\end{equation}
which is a polynomial function of $z$.
Then
\begin{equation}\label{F(r)}
F(r)=\int_{|w|=1}\vartheta_r(w)\psi_r(w)\phi_r(w)\,dw\,.
\end{equation}
\begin{lem}
\label{lemma:functions-z}
Let $z\in \C$ and $w\in \C^\times$. Then:
\medskip

\begin{center}
\begin{tabular}{llll}
\setlength{\extrarowheight}{.3em}
&$\psi_{-z}(w)=\psi_{z}(w)$, \qquad\qquad &$\psi_{z}(-w)=\psi_{z}(w)$, \qquad\qquad
&$\psi_{z}(iw)=\psi_{z}(w)$,\\[.2em]
&$\phi_{-z}(w)=\phi_{z}(w)$, \qquad &$\phi_{z}(-w)=-\phi_{z}(w)$, \qquad 
&$\phi_{z}(iw)=-i\phi_{z}(w)$,\\[.2em]
&$\vartheta_{-z}(w)=\vartheta_{z}(w)$, \qquad &$\vartheta_{z}(-w)=\vartheta_{z}(w)$, \qquad 
&$\vartheta_{z}(iw)=\vartheta_{z}(w)$.
\end{tabular}
\end{center}
\smallskip
\end{lem}
\begin{proof}
Set $\mu(z,w)=\bi \frac{z}{b}\cz(w) \beta_1 +\bi \frac{z}{b}\cz(-iw) \beta_2$, so that 
$\psi_z(w)=f\times \varphi_{\mu(z,w)}(y)$. Then $\mu(-z,w)$, $\mu(z,-w)$ and $\mu(z,iw)$
are transformed into $\mu(z,w)$ by sign changes and transposition of $\beta_1$ and of $\beta_2$.
The equalities for $\psi_z(w)$ then follow because the spherical function $\varphi_\lambda$ is 
$W$-invariant in the parameter $\lambda$.

The equalities for $\phi_{z}(w)$ are an immediate consequence of \eqref{eq:c-s} and the fact that the functions $\cz$, $\sz$ and $p_1q_1$ are odd. 

To prove the relations for $\vartheta_z(w)$, notice that 
$-\frac{m_\mrm}{2}+k = \frac{m_\mrm}{2}-h$ 
where $h=m_\mrm-k \in \{1,\dots,m_\mrm-1\}$ when $k \in \{1,\dots,m_\mrm-1\}$.
Hence 
$$
\prod_{k=1}^{m_\mrm-1} \Big[\big(-\frac{z}{b}\sz(w)-\frac{m_\mrm}{2}+k\big)^2+\frac{z^2}{b^2}\cz(w)^2 \Big] =
\prod_{h=1}^{m_\mrm-1}
\Big[ \big(\frac{z}{b}\sz(w)-\frac{m_\mrm}{2}+h\big)^2+\frac{z^2}{b^2}\cz(w)^2 \Big] \,.
$$
This proves the first two equalities for $\vartheta_z(w)$ since $\sz(-w)=-\sz(w)$. For the last equality, notice that
\begin{eqnarray*}
&&\big(\frac{z}{b}\sz(iw)-\frac{m_\mrm}{2}+k\big)^2+\frac{z^2}{b^2}\cz(iw)^2\\
&&\qquad =\big[\frac{z}{b}i\cz(w)-\frac{m_\mrm}{2}+k+i\frac{z}{b}i\sz(w)\big]
\big[\frac{z}{b}i\cz(w)-\frac{m_\mrm}{2}+k-i\frac{z}{b}i\sz(w)\big]\\
&&\qquad =\big[-\frac{z}{b}\sz(w)-\frac{m_\mrm}{2}+k+i\frac{z}{b}\cz(w)\big]
\big[\frac{z}{b}\sz(w)-\frac{m_\mrm}{2}+k-i\frac{z}{b}\cz(w)\big]\,.
\end{eqnarray*}
Hence, since $m_\mrm$ is even,
\begin{eqnarray*}
&&\prod_{k=1}^{m_\mrm-1}
\Big[ \big(\frac{z}{b}\sz(iw)-\frac{m_\mrm}{2}+k\big)^2+\frac{z^2}{b^2}\cz(iw)^2 \Big] \\
&&\qquad =
\prod_{k=1}^{m_\mrm-1}
\big[\frac{z}{b}\sz(w)-\frac{m_\mrm}{2}+k+i\frac{z}{b}\cz(w)\big] \cdot
 (-1)^{m_\mrm-1} \prod_{h=1}^{m_\mrm-1}
\big[\frac{z}{b}\sz(w)-\frac{m_\mrm}{2}+h-i\frac{z}{b}\cz(w)\big]\\
&&\qquad =
-\prod_{k=1}^{m_\mrm-1} \Big[ \big(\frac{z}{b}\sz(w)-\frac{m_\mrm}{2}+k\big)^2+\frac{z^2}{b^2}\cz(w)^2 \Big]\,. 
\end{eqnarray*}
This proves the claim because $\cz(w)^2-\cz(-iw)^2$ changes sign under the transformation
$w \to iw$.
\end{proof}

Thus  \cite[Lemma 3]{HPP15} generalizes as follows.
\begin{lem}
\label{lemma:holoextF}
The function $F(r)$, \eqref{F(r)}, extends holomorphically to
\begin{eqnarray}\label{2.2}
F(z)&=&\int_{|w|=1}\vartheta_z(w)\psi_z(w)\phi_z(w)\,dw\,,
\end{eqnarray}
where
\[
z\in \C\setminus i((-\infty,- L]\cup[L, +\infty))
\]
and $L$ is the constant defined in \eqref{eq:L}.
The function $F(z)$ is even and $F(z)z^{-2}$ is bounded near $z=0$.
\end{lem}

The following proposition, giving an initial holomorphic extension of the resolvent across the spectrum of the Laplacian, has been independently proven by Mazzeo and Vasy \cite[Theorem 1.3]{MV05} and by Strohmaier \cite[Proposition 4.3]{Str05} for general Riemannian symmetric spaces of the noncompact type and even rank. It shows that all possible resonances of the resolvent are located along the half-line $i(-\infty, -L]$. According to our conventions, we will omit $f$ and $y$ from the notation and write $R(z)$ instead of $[R(z)f](y)$. 

\begin{pro}
The resolvent $R(z)=[R(z)f](y)$ extends holomorphically from $\C \setminus \big((-\infty,0] \cup i(-\infty, -L]\big)$ to a logarithmic Riemann
surface branched along $(-\infty,0]$, with the preimages of  $i\big((-\infty, -L]\cup [L, +\infty)\big)$ removed and, in terms of monodromy, it satisfies the following equation
\begin{equation*}
R(z e^{2i\pi})=R(z)+2 i\pi\, F(z)\\ \qquad (z \in \C \setminus  \big((-\infty,0]\cup i(-\infty, -L]\cup i[L,+\infty) \big)).
\end{equation*}
\end{pro}

The starting point for studying the meromorphic extension of $R$ across $i(-\infty, -L]$ is the Proposition \ref{pro:holoextRF} below. It says that this meromorphic extension is equivalent to that of function $F$. This proposition is analogous to \cite[Proposition 4]{HPP15} and its proof is omitted.

\begin{pro}
\label{pro:holoextRF}
Fix $x_0>0$ and $y_0>0$. Let
\begin{eqnarray*}
Q&=&\{z\in \C; \Re z>x_0,\ y_0> \Im z \geq 0\}\\
U&=&Q\cup  \{z\in \C; \Im z <0\}
\end{eqnarray*}
Then there is a holomorphic function $H:U\to \C$ (depending on $f\in C^\infty_c(\X)$ and $y\in \X$, which are omitted from the notation) such that
\begin{equation}
\label{eq:holoextRF}
R(z)=H(z)+\pi i\, F(z) \qquad (z\in Q).
\end{equation}
As a consequence, the resolvent $R(z)=[R(z)f](y)$ extends holomorphically from $\C^+$ to
$\C \setminus \big((-\infty,0] \cup i(-\infty, -L]\big)$.
\end{pro}

\subsection{Meromorphic extension and residue computations}

This section is devoted to the meromorphic extension of the function $F$ (and hence of the resolvent)  across the half-line $i(-\infty, -L]$. We set 
\begin{equation}
\label{eq:psi-theta}
\psi^\vartheta_z(w)=\vartheta_z(w)\psi_z(w)
\end{equation}
 and follow the stepwise extension procedure for $F$ from \cite[\S2 and \S 3]{HPP15} with $\psi_z(w)$ replaced by $\psi^\vartheta_z(w)$. Some formulas are simplified by the fact that we are only dealing with the special case of $\X_1=\X_2$ with $\beta_1$ and $\beta_2$ of equal norms $b_1=b_2=b$ and equal odd multiplicities $m_{\beta_1}=m_{\beta_2}$. Notice also that in this paper, studying the singularities of the Plancherel density, we are replacing the elements $\rho_{\beta_1}$ and $\rho_{\beta_2}$ used in \cite{HPP15} with $\wt \rho_{\beta_1}$ and $\wt \rho_{\beta_2}$, which are equal and have value $\frac{1}{2}\big( m_\rl+\frac{m_\rs}{2}\big)$. Indeed, in the case of direct product of rank one symmetric spaces treated in \cite{HPP15},  there was no need of introducing multiple notation by distinguishing between $\rho_\beta$ and $\wt \rho_\beta$ for $\beta \in \Sigma_*$. The distinction is now necessary since $\rho_{\beta_1}=\wt \rho_{\beta_1}=\wt \rho_{\beta_2}\neq \rho_{\beta_2}$.
Furthermore, we omit the index $j$ from the notation used in \cite{HPP15} when it only refers to which of the two factors one considers. So, for instance \cite[(38)]{HPP15} yields, for the set of singularities of the product $p_1 q_1$ from \eqref{eq:pq}, the set
\begin{equation}
\label{eq:Sj}
S=S_+ \cup (-S_+)\,,
\end{equation}
where
\begin{equation}
\label{eq:Splus}
S_+=ib(\wt\rho_{\beta_1}+\Bbb Z_{\geq 0})=ib\Big(\frac{1}{2}\big(m_\rl+\frac{m_\rs}{2}\big)+\Bbb Z_{\geq 0}\Big).
\end{equation}
For $r>0$ and $c,d \in \R\setminus \{0\}$ recall the sets
\begin{eqnarray*}
\Dg_r&=&\{z\in \C;\ |z|<r\}\,,\\
\Eg_{c,d}&=&\big\{\xi+i\eta\in \C;\ \left(\tfrac{\xi}{c}\right)^2+ \left(\tfrac{\eta}{d}\right)^2<1\big\}\,.
\end{eqnarray*}
and the role they play in \cite[\S 1.4]{HPP15} for the functions $\sz$ and $\cz$ introduced in \eqref{eq:c-s}. Then \cite[Prop. 6]{HPP15} translates in the following proposition.

\begin{pro}\label{2.5}
Suppose $z\in \C\setminus i((-\infty,- L]\cup[L, \infty))$ and $r>0$ are such that
\begin{equation}\label{2.5.1}
S\cap z\partial E_{\cz(r),\sz(r)}=\emptyset.
\end{equation}
Then
\begin{equation}
\label{eq:F-contour}
F(z)=F_r(z)+2\pi i \, F_{r,{\rm res}}(z),
\end{equation}
where
\begin{eqnarray*}
F_r(z)&=&\int_{\partial D_r}\psi^\vartheta_z(w)\phi_z(w)\,dw,\\
 F_{r,{\rm res}}(z)&=&{\sum}_{w_0}'\psi^\vartheta_z(w_0)\Res_{w=w_0}\phi_z(w),
\end{eqnarray*}
and $\sum_{w_0}'$ denotes the sum over all the $w_0$ such that
\begin{equation}\label{2.5.2}
z\cz(w_0)\in S\cap z(\Eg_{\cz(r),\sz(r)}\setminus [-1,1])
\end{equation}
or
\begin{equation}\label{2.5.3}
z\cz(-iw_0)\in S\cap z(\Eg_{\cz(r),\sz(r)}\setminus [-1,1]).
\end{equation}
Both $F_r$ and $ F_{r,{\rm res}}$ are holomorphic functions on the open subset of $\C\setminus i((-\infty,- L]\cup[L, \infty))$ where the condition \eqref{2.5.1} holds. Furthermore, $F_r$ extends to a holomorphic function on the open subset of $\C$ where the condition \eqref{2.5.1} holds.
\end{pro}

To make the function $ F_{r,{\rm res}}(z)$ explicit, we proceed as in \cite[\S 3.1]{HPP15}. The present situation is in fact simpler, because only the case
$L_{1,\ell}=L_{2,\ell}$ occurs. We denote this common value by $L_\ell$, i.e. 
we define for $\ell\in\Zb_{\geq 0}$
\begin{equation}
\label{eq:Lell}
L_\ell=b(\wt \rho_{\beta_1}+\ell)=b\big(\tfrac{m_\rl}{2}+\tfrac{m_\rs}{4}+\ell\big)\,.
\end{equation}
So $S_+=\{iL_\ell; \ell\in \Zb_{\geq 0}\}$.

If $0\neq z\in \C\setminus i \big((-\infty,-L_{\ell}]\cup [L_{\ell},+\infty)\big)$, then $
\frac{i L_\ell}{z} \in \C\setminus [-1,1]$ and we can 
uniquely define $w_1^\pm \in \Dg_1\setminus \{0\}$ satisfying
\begin{eqnarray}
\label{eq:cw1pm-eps-k1}
z\cz(w_1^\pm)&=&\pm i L_{\ell}\,.
\end{eqnarray}
Since $c(-w)=-c(w)$, we obtain that $w_1^-=-w_1^+$. Moreover, 
$w_1$ satisfies \eqref{2.5.2} if and only if $w_2=iw_1$ satisfies \eqref{2.5.3}
because $z(\Eg_{\cz(r),\sz(r)}\setminus [-1,1])$ is symmetric with respect to the origin $0\in\C$. 
Hence
\begin{multline}
\label{eq:Gr-1}
 F_{r,{\rm res}}(z)={\sum}_{w_1^+}' 
\Big[\psi^\vartheta_z(w_1^+)\Res_{w=w_1^+}\phi_z(w)+ 
\psi^\vartheta_z(w_1^-)\Res_{w=w_1^-}\phi_z(w)+\\
\psi^\vartheta_z(iw_1^+)\Res_{w=iw_1^+}\phi_z(w)+ 
\psi^\vartheta_z(iw_1^-)\Res_{w=iw_1^-}\phi_z(w)\Big]\,,  
\end{multline}
where $\sum_{w_1^+}'$ denotes the sum over all the $w_1^+$ such that
$
z\cz(w_1^+)\in S_+\cap z(\Eg_{\cz(r),\sz(r)}\setminus [-1,1])
$
and $w_1^-=-w_1^+$.

Then, using Lemma \ref{lemma:functions-z}, we obtain the following analogue of \cite[Lemma 9]{HPP15}.

\begin{lem}
\label{lemma:Gjkj}
For $\ell\in \Zb_{\geq 0}$ and
 $0\neq z\in \C\setminus i \big((-\infty,-L_{\ell}]\cup [L_{\ell},+\infty)\big)$,
let $w_1^\pm$ be defined by \eqref{eq:cw1pm-eps-k1}. Then 
\begin{eqnarray}
\label{eq:psithetazw1}
&&\psi^\vartheta_z(w_1^+)=\psi^\vartheta_z(w_1^-)=\psi^\vartheta_z(iw_1^+)=
\psi^\vartheta_z(iw_1^-)\,,\\
\label{eq:Res-phiw1}
&&
\Res_{w=w_1^+} \phi_z(w)=\Res_{w=w_1^-} \phi_z(w)=
\Res_{w=iw_1^+} \phi_z(w)=\Res_{w=-iw_1^-} \phi_z(w)\,.\\
\end{eqnarray}
Explicitly, 
\begin{eqnarray*}
\label{eq:psithetazw1-1}
\psi^\vartheta_z(w_1^+)&=&\psi^\vartheta_z\Big(\cz^{-1}\big(\tfrac{iL_{\ell}}{z}\big)\Big)\,,\\
\label{eq:Res-phiw1}
\Res_{w=w_1^+} \phi_z(w)&=&-C_{\ell}\,
p_1\Big(i z (\sz\circ \cz^{-1})\big(\tfrac{iL_{\ell}}{z}\big)\Big) \,
q_1\Big(i z (\sz\circ \cz^{-1})\big(\tfrac{iL_{\ell}}{z}\big)\Big)\,,
\end{eqnarray*}
where
\begin{equation}\label{constants}
C_{\ell}= \frac{b}{\pi}L_{\ell}\, p_1(iL_{\ell}) \neq 0\,.
\end{equation}
\end{lem}

\begin{cor}
\label{cor:Gell}
Let $\ell\in \Zb_{\geq 0}$ and $0\neq z\in \C\setminus i \big((-\infty,-L_{\ell}]\cup [L_{\ell},+\infty)\big)$. Set
\begin{eqnarray}
\label{eq:Gell}
G_{\ell}(z)&=&-C_\ell \psi^\vartheta_z\Big(\cz^{-1}\big(\tfrac{iL_{\ell}}{z}\big)\Big)
p_1\Big(i z (\sz\circ \cz^{-1})\big(\tfrac{iL_{\ell}}{z}\big)\Big) \,
q_1\Big(i z (\sz\circ \cz^{-1})\big(\tfrac{iL_{\ell}}{z}\big)\Big)\,,\\
\label{eq:Sjrz}
\Sg_{r,z,+}&=&\{\ell\in \Zb_{\geq 0}: i L_{\ell}\in z\big(\Eg_{\cz(r),\sz(r)}\setminus [-1,1]\big)\}\,.
\end{eqnarray}
Then the function $ F_{r,{\rm res}}(z)$ on the right-hand side of \eqref{eq:F-contour} is given by
\begin{equation}
\label{eq:Gr-Splus}
 F_{r,{\rm res}}(z)=4 \sum_{\ell \in \Sg_{r,z,+}} G_\ell(z)\,.
\end{equation}
\end{cor}

The following proposition is analogous to \cite[Proposition 10]{HPP15}.

\begin{pro}
\label{prop:SjrzW}
For $0<r<1$ and $0\neq z\in \C\setminus i((-\infty,- L]\cup[L, +\infty))$,
let $\Sg_{r,z,+}$ be as in \eqref{eq:Sjrz}. Moreover,
let $W \subseteq \C$ be a connected open set such that
\begin{equation}
\label{eq:intersection-W}
S \cap W \partial \Eg_{\cz(r),\sz(r)}=\emptyset
\end{equation}
and set 
\begin{equation}
\label{eq:SjrzW}
S_{r,W,+}=\{\ell\in \Zb_{\geq 0}: i L_{\ell}\in W\Eg_{\cz(r),\sz(r)}\}
\qquad (z \in W\setminus i\R)\,.
\end{equation}
Then $S_{r,z,+}=S_{r,W,+}$. In particular,  $S_{r,z,+}$ does not depend on 
$z \in W\setminus i\R$.
\end{pro}

Proceeding now as in \cite[Corollaries 11, 13 and Lemma 12]{HPP15}, we obtain the following 
result for points on $i\R$.

\begin{cor}
\label{cor:SjrvW}
For every $iv \in i\R$ and for every $r$ with $0<r<1$ and $vc(r)\notin iS$
there is a connected open neighborhood $W_v$ of $iv$ in $\C$ satisfying the following conditions.
\begin{enumerate}
\item
$S \cap W_v \partial \Eg_{\cz(r),\sz(r)}=\emptyset$\,,
\item
$S_{r,W_v,+}=\{\ell\in \Zb_{\geq 0}:i L_{\ell}\in iv\Eg_{\cz(r),\sz(r)}\}=
\leftFpar 0, N_{v}\rightFpar$ for some $N_{v}\in \Zb_{\geq 0}$.
\item
For $n\in \Zb_{\geq 0}$, set 
\begin{equation}
\label{eq:Ijm}
I_n=b\wt\rho_{\beta_1}+b[n,n+1)=[L_{n},L_{n+1})\,.
\end{equation}
If $v \in I_n$ then $N_{v}=n$. Hence
\begin{equation}
\label{eq:Gr-Gell}
 F_{r,{\rm res}}(z)=4 \sum_{\ell=1}^n G_{\ell}(z) \qquad (z \in W_v\setminus i\R)\,.
\end{equation}
\end{enumerate}
\end{cor}

We recall the relevant Riemann surfaces from \cite[(76)]{HPP15}. Fix  $\ell\in\Bbb Z_{\geq 0}$. Then
\begin{equation}
\label{eq:Mjk}
\M_{\ell}=\Big\{(z,\zeta)\in\C^\times\times(\C\setminus\{i, -i\})\ :\ \zeta^2=\Big(\frac{iL_{\ell}}{z}\Big)^2-1\Big\}
\end{equation}
is a Riemann surface above $\C^\times$, with projection map $\pi_{\ell}:\M_{\ell} \ni (z,\zeta)\to z\in \C^\times$\,.
The fiber of $\pi_{\ell}$ above $z\in \C^\times$ is $\{(z,\zeta), (z,-\zeta)\}$.
In particular, the restriction of $\pi_{\ell}$ to $\M_{\ell}\setminus \{(\pm i L_{\ell},0)\}$ is a double cover
of $\C^\times \setminus \{\pm i L_{\ell}\}$.

Now \cite[Lemma 15]{HPP15} has the following analogue. 
The difference is that we have replaced $\psi_z(w)$ by $\psi^\vartheta_z(w)$.
So we have to look for possible cancellations of singularities arising from the additional polynomial factor $\vartheta_z$.

\begin{lem}\label{2.8}
In the above notation,
\begin{eqnarray}\label{eq:tGell}
\wt G_{\ell}:\M_{\ell}\ni (z,\zeta)&\to&\frac{b}{\pi} \, L_{\ell}\,  p_1(i L_{\ell})\psi^\vartheta_z\big(\frac{i L_{\ell}}{z}-\zeta\big)
 p_1(iz\zeta) q_1(iz\zeta)\in\C
\end{eqnarray}
is the meromorphic extension to $\M_{\ell}$ of a lift of $G_{\ell}$.

The function $\wt G_{\ell}$ has simple poles at all points $(z,\zeta)\in \M_{\ell}$ such that
\begin{equation}\label{2.9.1}
z=\pm i \sqrt{ L_{\ell}^2+L_{m}^2},
\end{equation}
where $m\in\Zb_{\geq 0} \setminus \bigleftFpar \ell -\big(\frac{m_\mrm}{2}-1\big), 
\ell +\big(\frac{m_\mrm}{2}-1\big)\bigrightFpar$.
\end{lem}
\begin{proof}
Formula \eqref{eq:tGell} is obtained using the lifts of $\cz^{-1}$ and $\sz\circ \cz^{-1}$,
as in \cite[Lemma 15]{HPP15}.  

The poles of $\wt G_{\ell}$ are the points $(z,\zeta)\in \M_\ell$
for which the function $\vartheta_z\big(\frac{i L_{\ell}}{z}-\zeta\big)
 p_1(iz\zeta) q_1(iz\zeta)$ is singular, i.e. the points for which $p_1(iz\zeta) q_1(iz\zeta)$ is singular and  $\vartheta_z\big(\frac{i L_{\ell}}{z}-\zeta\big)\neq 0$. By construction, 
 $p_1(iz\zeta) q_1(iz\zeta)$ is singular if and only if $iz\zeta\in S$, see \eqref{eq:Sj}. In this case, 
there exist $\epsilon\in \{\pm 1\}$ and $m\in \Zb$ so that $\zeta=\frac{\epsilon L_m}{z}$. Hence $\zeta^2=\frac{L_m^2}{z^2}$. Since $(z,\zeta)\in \M_\ell$, we also have 
$\zeta^2=-\frac{L_\ell^2}{z^2}-1$. Thus $z=\pm i \sqrt{ L_{\ell}^2+L_{m}^2}$.

We now compute $\vartheta_z\big(\frac{i L_{\ell}}{z}-\zeta\big)$ for such $(z,\zeta)$.
Set 
$$
w=\frac{i L_{\ell}}{z}-\zeta=\frac{i L_{\ell}}{z}-\frac{\epsilon L_m}{z}=
\frac{i L_{\ell}-\epsilon L_m}{\pm i \sqrt{ L_\ell^2+ L_m^2}}\,.
$$
Then 
$$
w^{-1}=\frac{\pm i \sqrt{L_\ell^2+ L_m^2}}{i L_{\ell}-\epsilon L_m}=
\frac{i L_{\ell}+\epsilon L_m}{\pm i \sqrt{L_\ell^2+ L_m^2}}\,.
$$
So,
\begin{eqnarray*}
\cz(w)&=&\frac{w+w^{-1}}{2}=\frac{L_{\ell}}{\pm \sqrt{ L_\ell^2+ L_m^2}}\,,\\
\cz(-iw)&=&\frac{w-w^{-1}}{2i}=\frac{\epsilon L_m}{\pm \sqrt{ L_\ell^2+ L_m^2}}\,.
\end{eqnarray*}
Hence
$$
z\cz(w)=iL_\ell\,, \qquad z\cz(-iw)=i\epsilon L_m\,, \qquad z\sz(w)=-\epsilon L_m\,.
$$
Substituting in \eqref{eq:vathetaz}, we obtain
\begin{equation}
\label{eq:vartheta-lm-1}
  \vartheta_z(w)
    =\frac{\big(L_m^2-L_\ell^2\big)}{b^2}\prod_{k=1}^{m_\mrm-1} \Big[ \big(\frac{-\epsilon L_m}{b}-\frac{m_\mrm}{2}+k\big)^2-\frac{L_\ell^2}{b^2}\Big]\,.
\end{equation}
The same argument used in the proof of Lemma \ref{lemma:functions-z} shows that the right-hand side of this equation is independent of $\epsilon\in \{\pm 1\}$. Using \eqref{eq:Lell}, we therefore obtain
\begin{eqnarray*}
\vartheta_z(w)&=& 
\big((\wt\rho_{\beta_1}+m)^2-(\wt \rho_{\beta_1}+\ell)^2\big)\prod_{k=1}^{m_\mrm-1} \big[ \big(\wt\rho_{\beta_1}+m-\frac{m_\mrm}{2}+k\big)^2-(\wt\rho_{\beta_1}+\ell)^2\big]\\
&=&(m-\ell)(m+\ell+2\wt\rho_{\beta_2}) 
  \prod_{k=1}^{m_\mrm-1} 
 \Big(m-\ell-\frac{m_\mrm}{2}+k\Big) \Big(m+\ell+2\wt\rho_{\beta_1}-\frac{m_\mrm}{2}+k\Big)\\
&=&(m-\ell)(m+\ell+2\wt\rho_{\beta_2}) 
  \prod_{h=-(\frac{m_\mrm}{2}-1)}^{\frac{m_\mrm}{2}-1} 
 \Big(m-\ell+h\Big) \Big(m+\ell+2\wt\rho_{\beta_1}+h\Big)\,.
\end{eqnarray*}
The values of $m\in \Zb_{\geq 0}$ making this polynomial vanish are:
\begin{eqnarray}
\label{eq:m=l}
&&m=\ell\,, \\
\label{eq:otherm}
&&m\in \Zb_{\geq 0} \cap \bigleftFpar \ell -\big(\tfrac{m_\mrm}{2}-1\big), 
\ell +\big(\tfrac{m_\mrm}{2}-1\big)\bigrightFpar\,,\\
&&m\in \Zb_{\geq 0} \cap \bigleftFpar -\ell-2\wt\rho_{\beta_1} -\big(\tfrac{m_\mrm}{2}-1\big),   -\ell -2\wt\rho_{\beta_1} + \big(\tfrac{m_\mrm}{2}-1\big)\bigrightFpar\,,
\end{eqnarray}
Observe that $-\ell -2\wt\rho_{\beta_1} + \big(\frac{m_\mrm}{2}-1\big)\geq 0$ if and only if
$(0\leq) \ell \leq  -2\wt\rho_{\beta_1} + \big(\frac{m_\mrm}{2}-1\big)$. Looking at the first two rows of Table \ref{table tilde rho}, we see that
this can happen if and only if $\G=\SO_0(p,2)$ with even $p\geq 6$. In this case, 
\begin{eqnarray*}
\bigleftFpar\ell -\big(\tfrac{m_\mrm}{2}-1\big),  
\ell +\big(\tfrac{m_\mrm}{2}-1\big)\bigrightFpar&=&\bigleftFpar\ell+2-\tfrac{p}{2}, \ell -2+\tfrac{p}{2}\bigrightFpar\\
\bigleftFpar -\ell-2\wt\rho_{\beta_1} -\big(\tfrac{m_\mrm}{2}-1\big),  
-\ell -2\wt\rho_{\beta_1} + \big(\tfrac{m_\mrm}{2}-1\big)\bigrightFpar
&=&\bigleftFpar -\ell+1-\tfrac{p}{2},-\ell-3+\tfrac{p}{2}\bigrightFpar.
\end{eqnarray*} 
Hence 
$$\Zb_{\geq 0} \cap \bigleftFpar -\ell+1-\tfrac{p}{2}, -\ell-3+\tfrac{p}{2}\bigrightFpar =
\bigleftFpar 0,-\ell-3+\tfrac{p}{2}\bigrightFpar$$
does not add zeros to those in \eqref{eq:otherm}.
In fact, $-\ell-3+\tfrac{p}{2}\leq \ell-2+\tfrac{p}{2}$ and, if $-\ell-3+\tfrac{p}{2}\geq 0$, i.e. 
$\ell\leq \tfrac{p}{2}-3$, then $\ell+2-\tfrac{p}{2}\leq 0$.
\end{proof}

For $\ell, m\in \Zb_{\geq 0}$, set
\begin{equation}
\label{eq:z-lm}
z_{\ell,m}=i \sqrt{L_\ell^2+L_m^2}
\end{equation}
and
\begin{equation}\label{2.9.5'}
\zeta_{\ell,m}= i \sqrt{\frac{L_m^2}{L_\ell^2+L_m^2}}\,.
\end{equation}
Let $\epsilon\in \{\pm 1\}$.
Then all points $(\pm z_{\ell,m},\epsilon\zeta_{\ell,m})$ are in $\M_\ell$. 
Open neighborhoods in $\M_\ell$ of these points
are the sets 
\begin{equation}
\label{eq:U1pm}
\Ug_{\ell,\pm}=\{(z,\zeta)\in \M_{\ell}\ ;\ \pm\Im z >0\}\,.
\end{equation}
and local charts on them are 
\begin{equation}\label{charts}
\kappa_{\ell,\pm}:\Ug_{\ell,\pm}\ni (z,\zeta)\to \zeta\in \C\setminus  i\big((-\infty, -1] \cup [1,+\infty)\big)\,, 
\end{equation}
inverted by setting $z=\pm i \,\frac{L_\ell}{\sqrt{\zeta^2+1}}$.

\begin{lem}\label{lem:chart-expressions}
The local expressions for $\wt G_\ell$ in terms of the charts (\ref{charts}) are
\begin{eqnarray}
\label{expression in terms of charts 1}
\big(\wt G_\ell\circ\kappa_{\ell,\pm}^{-1}\big)(\zeta) 
=\pm \frac{b}{\pi} \, L_{\ell}\,  p_1(i L_\ell)
p_2\Big(\tfrac{L_\ell\zeta}{\sqrt{\zeta^2+1}}\Big) q_1\Big(\tfrac{L_\ell\zeta}{\sqrt{\zeta^2+1}}\Big)
\psi^\vartheta_{i\,\frac{L_\ell}{\sqrt{\zeta^2+1}}}\big(\sqrt{\zeta^2+1}\mp \zeta\big)\,.
\end{eqnarray}
Suppose $m\in\Zb_{\geq 0} \setminus \bigleftFpar \ell -\big(\frac{m_\mrm}{2}-1\big), 
\ell +\big(\frac{m_\mrm}{2}-1\big)\bigrightFpar$. Then 
the residue of the local expression of $\wt G_\ell$ at a point $(z,\zeta)\in \M_\ell$ with 
$z=\pm z_{\ell,m}$ is
\begin{equation}\label{2.11.1}
\Res_{\zeta=\pm \zeta_{\ell,m}} (\wt G_\ell\circ \kappa_{\ell,\pm}^{-1})(\zeta)
=\pm \frac{1}{i\pi^2} C_{\ell,m} (f\times \varphi_{\frac{L_\ell \beta_1 +L_m\beta_2}{b}})(y)\,.
\end{equation}
In \eqref{2.11.1},
\begin{equation}
\label{eq:const-lm}
C_{\ell,m}=b L_\ell\,  p_1(i L_\ell) p_2(i L_m) \vartheta_0\Big(\frac{L_\ell}{b},\frac{L_m}{b}\Big)\,,
\end{equation}
where $\vartheta_0$ is as in \eqref{eq:vartheta0}, is a positive constant.
\end{lem}
\begin{proof}
The computation of the residues is as in \cite[Lemma 16]{HPP15}. 
The constant  $\vartheta_0\big(\frac{L_\ell}{b},\frac{L_m}{b}\big)$ agrees with \eqref{eq:vartheta-lm-1} with $(z,\zeta)=(z_{\ell,m},\epsilon \zeta_{\ell,m})$,
and we only need to prove that it is positive. Recall that \eqref{eq:vartheta-lm-1} is independent of $\epsilon$. Hence
\begin{equation}
\label{eq:vartheta-lm-2}
 \vartheta_0\Big(\frac{L_\ell}{b},\frac{L_m}{b}\Big)
    =\frac{\big(L_m^2-L_\ell^2\big)}{b^2}\prod_{k=1}^{m_\mrm-1} 
    \Big(\frac{L_m}{b}-\big(\frac{m_\mrm}{2}-k\big)-\frac{L_\ell}{b}\Big)
    \Big(\frac{L_m}{b}-\big(\frac{m_\mrm}{2}-k\big)+\frac{L_\ell}{b}\Big)\,.
\end{equation}
If $m> \ell+(\frac{m_\mrm}{2}-1\big)\geq \ell$, then all factors appearing in the above product are positive. If $m< \ell-(\frac{m_\mrm}{2}-1\big)\leq \ell$, then all factors $\frac{L_m}{b}-\big(\frac{m_\mrm}{2}-k\big)+\frac{L_\ell}{b}$ are positive, whereas $L_m^2-L_\ell^2$ as well as the $m_\mrm-1$ factors $\frac{L_m}{b}-\big(\frac{m_\mrm}{2}-k\big)-\frac{L_\ell}{b}$
are negative. Since $m_\mrm$ is even, we conclude that $\vartheta_0\big(\frac{L_\ell}{b},\frac{L_m}{b}\big)>0$ in all cases. 
\end{proof}

A different parametrization of the singularities of $\wt G_\ell$ will turn out to be more convenient. 
Observe first that, by \eqref{eq:rho} and \eqref{eq:wtrhob}, 
$$\wt \rho_{\beta_1}=\rho_{\beta_1}=\rho_{\beta_2}-\frac{m_\mrm}{2}\,.$$
We will use the following notation for $(\ell_1,\ell_2)\in \Zb_{\geq 0}^2$:
\begin{equation}
\label{eq:lambdal1l2}
\lambda(\ell_1,\ell_2)=(\rho_{\beta_1}+\ell_1)\beta_1+(\rho_{\beta_2}+\ell_2)\beta_2=
\frac{1}{b}\big( L_{\ell_1} \beta_1+L_{\ell_2+\frac{m_\mrm}{2}}\beta_2\big).
\end{equation}  

\begin{cor}
\label{cor:singGell}
Keep the notation of Lemma \ref{lem:chart-expressions}.
If $\ell\in\leftFpar 0,\frac{m_\mrm}{2}-1\rightFpar$, then $\wt G_\ell$ has simple poles at the points $(z,\zeta)\in \M_\ell$ with $z=\pm i|z|$ and
\begin{equation}
\label{eq:singular1}
b^{-2}|z|^2=(\rho_{\beta_1}+\ell)^2 +(\rho_{\beta_2}+\ell+k)^2 \qquad (k\in \Zb_{\geq 0})\,. 
\end{equation}
If $\ell\in\frac{m_\mrm}{2}+\Zb_{\geq 0}$, then $\wt G_\ell$ has simple poles at the points $(z,\zeta)\in \M_\ell$ with $z=\pm i|z|$ and satisfying either \eqref{eq:singular1} or
\begin{equation}
\label{eq:singular2}
b^{-2}|z|^2=(\rho_{\beta_1}+m)^2 +(\rho_{\beta_2}+\ell_0)^2 \qquad 
(m\in\leftFpar 0,\ell_0\rightFpar )\,,
\end{equation}
where $\ell_0=\ell-\frac{m_\mrm}{2}$.

The residue of the local expression of $\wt G_\ell$ at a point $(z,\zeta)\in \M_\ell$ with $z=\pm i |z|$ satisfying \eqref{eq:singular1} is 
\begin{equation}
\label{eq:resGell1}
\Res_{\zeta=\pm \zeta_{\ell,\ell+\frac{m_\mrm}{2}+k}} (\wt G_\ell\circ \kappa_{\ell,\pm}^{-1})(\zeta)
=\pm \frac{1}{i\pi^2} C_{\ell,\ell+\frac{m_\mrm}{2}+k} \big(f\times \varphi_{\lambda(\ell,\ell+k)}\big)(y)\,.
\end{equation}
The residue of the local expression of $\wt G_\ell$ at a point $(z,\zeta)\in \M_\ell$ with $z=\pm i |z|$ satisfying \eqref{eq:singular2} is 
\begin{equation}
\label{eq:resGell2}
\Res_{\zeta=\pm \zeta_{\ell,m}} (\wt G_\ell\circ \kappa_{\ell,\pm}^{-1})(\zeta)
=\pm \frac{1}{i\pi^2} C_{\ell,m} \big(f\times \varphi_{\lambda(m,\ell_0)}\big)(y)\,.
\end{equation}
\end{cor}
\begin{proof}
We have $\ell\in\leftFpar 0,\frac{m_\mrm}{2}-1\rightFpar$ if and only if $0\in
\leftFpar \ell-\big(\frac{m_\mrm}{2}-1\big),\ell+\big(\frac{m_\mrm}{2}-1\big)\rightFpar$. 
In this case, $m\in \Zb_{\geq 0}\setminus \leftFpar \ell-\big(\frac{m_\mrm}{2}-1\big),\ell+\big(\frac{m_\mrm}{2}-1\big)\rightFpar=\ell + \frac{m_\mrm}{2}+ \Zb_{\geq 0}$
is of the form $m= \ell + \frac{m_\mrm}{2}+k$ with $k\in \Zb_{\geq 0}$. Hence
$\frac{L_\ell}{b}=\rho_{\beta_1}+\ell$ and $\frac{L_m}{b}=\wt \rho_{\beta_1}+\frac{m_\mrm}{2}+ \ell+k=\rho_{\beta_2}+\ell+k$. 

On the other hand, if  $\ell\in\frac{m_\mrm}{2}+\Zb_{\geq 0}$ and $m\in \Zb_{\geq 0}\setminus \leftFpar \ell-\big(\frac{m_\mrm}{2}-1\big),\ell+\big(\frac{m_\mrm}{2}-1\big)\rightFpar$, then 
either $m\in \ell + \frac{m_\mrm}{2}+ \Zb_{\geq 0}$ (and the above applies), or $m \in 
\leftFpar 0,\ell_0\rightFpar$. In the latter case, $\frac{L_\ell}{b}= \wt\rho_{\beta_1}+\frac{m_\mrm}{2}+\ell_0=\rho_{\beta_2}+\ell_0$ and 
$\frac{L_m}{b}=\rho_{\beta_1}+m$. Observe also that 
$\varphi_{\lambda(\ell_0,m)}=\varphi_{\lambda(m,\ell_0)}$
by $W$-invariance.
\end{proof}

\begin{figure}
\includegraphics[trim = 10mm 100mm 10mm 100mm, clip, width=16cm]{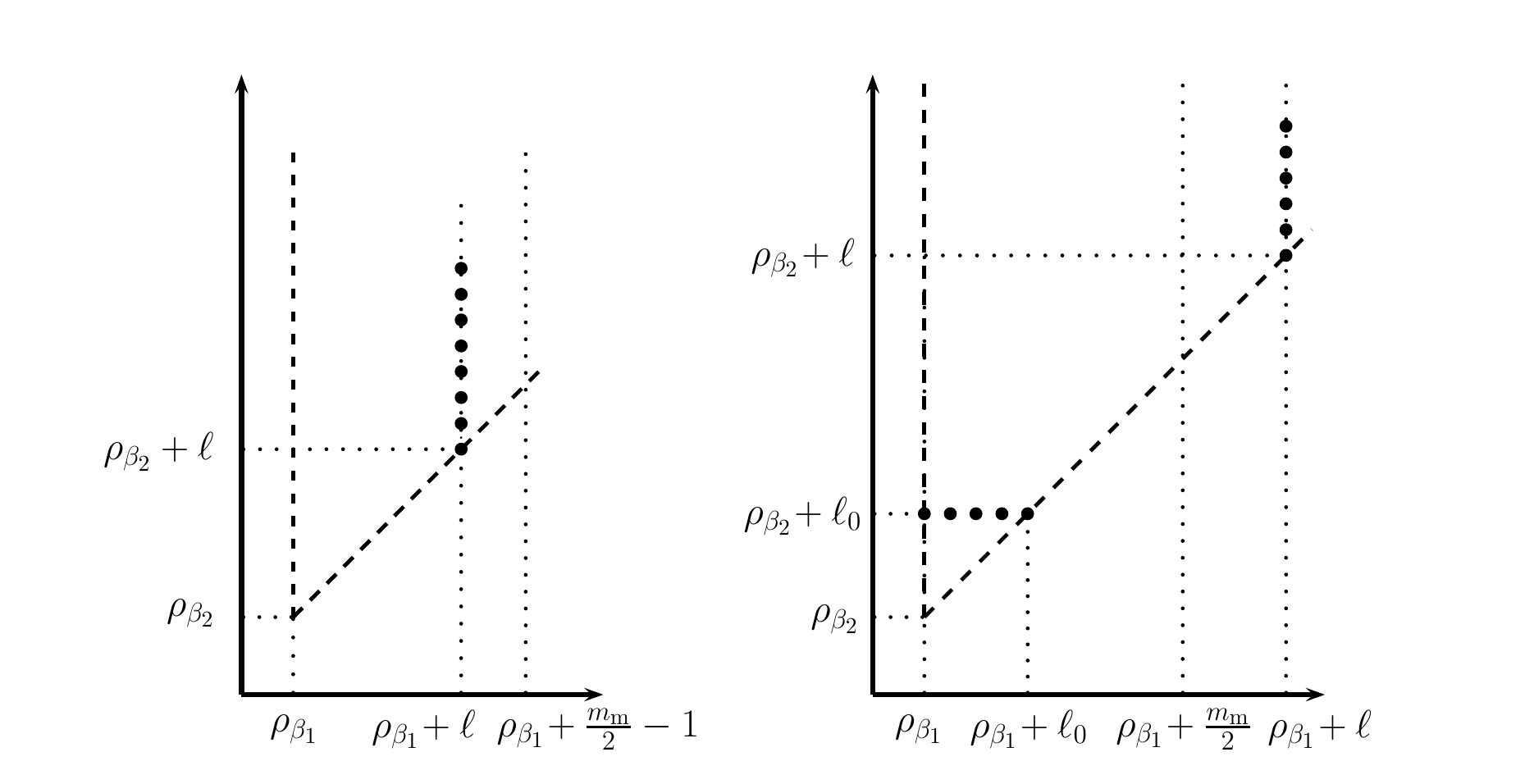} 

\caption{On the left: $\lambda(\ell,\ell+k)$ for $\ell\in \leftFpar 0,\frac{m_{\rm m}}{2}\rightFpar$. 
On the right: $\lambda(\ell,\ell+k)$ and $\lambda(m,\ell_0)$ for $\ell\geq \frac{m_{\rm m}}{2}$}

\end{figure}


We now proceed with the piecewise extension of $F$ along the negative imaginary half-line 
$-i[L,+\infty)$.  Recall from Corollary \ref{cor:SjrvW} that for $v\in I_{n}=[L_{n},L_{n+1})$ with $n\in \Zb_{\geq 0}$ there exists
$0< r_v <1$ and an open neighborhood $W_v$ of $-iv$ in $\C$ such that
\begin{equation}
\label{eq:Fresidues1}
F(z)=F_{r_v}(z)+ 4\sum_{\ell=0}^{n} G_{\ell}(z)
\qquad (z\in W_v\setminus i\R)\, ,
\end{equation}
where the function $F_{r_v}$ is holomorphic in $W_v$. This equality extends then to $I_{-1}=(0,L)$ by allowing empty sums. By possibly shrinking $W_v$, we may also assume that $W_v$ is an open disk around $-iv$ such that
$$
W_v\cap i\R \subseteq
\begin{cases}
  -iI_{n} &\text{for $v\in I_n^\circ$},\\
  -i(I_{n}-\frac{b}{2}) &\text{for $v=L_n$}\,.
\end{cases}
$$
In addition, for $0<v < L$ we define $W_v$ to be an open ball around $-iv$ in $\C$ such that $W_{v}\cap i\R\subset (0,L)$.
If $v\in I_n$, $v'\geq L$ and $W_v\cap W_{v'}\not=\emptyset$, then we obtain for $z\in W_v\cap W_{v'}$
$$
F_{r_{v'}}(z)=F_{r_v}(z)+
\begin{cases}
  0& \text{if $v'\in I_n$},\\
  4\, G_n(z) &\text{if $v'\in I_{n-1}$}. 
\end{cases}
$$
Now we set
\begin{eqnarray*}
W_{(-1)}=\bigcup_{v\in I_{-1}} W_v  \quad &\text{and} & \quad
W_{(n)}=\bigcup_{v\in I_n} W_v \qquad (n\in \Zb_{\geq 0})\,.
\end{eqnarray*}
For $n\in\Zb_{\geq 1}$ we define a holomorphic function $F_{(n)}: W_{(n)}\to \C$ by
\[
F_{(n)}(z)= \begin{cases}
F_{r_v}(z) &\text{if $n\in \Zb_{\geq 0}$, $v\in I_n$ and $z\in W_v$}\\
F(z) &\text{if $n=-1$ and $z\in W_{(-1)}$}\,.
\end{cases}
\]
We therefore obtain the following analogue of \cite[Proposition 18]{HPP15}.

\begin{pro}
\label{cor:FextendedMainSection}
For every integer $n\in \Zb_{\geq -1}$ we have
 \begin{equation}
\label{eq:Fresidues2}
F(z)=F_{(n)}(z)+4\sum_{\ell=0}^{n} G_{\ell}(z)
\qquad (z\in W_{(n)}\setminus i\R)\,,
\end{equation}
where $F_{(n)}$ is holomorphic in $W_{(n)}$, the $G_\ell$ are as in \eqref{eq:Gell}, and empty sums are defined to be equal to $0$.
\end{pro}

We can continue $F$ across $-i(0,+\infty)$ inductively, as in the case of the direct product of two rank one symmetric spaces in \cite{HPP15}. Our specific case $\X_1=\X_2$ is slightly easier, as for instance one gets just one regularly spaced sequence of branching points $L_\ell$. Since the procedure does not involve new steps, we will limit ourself to overview the different parts and state the final result, referring the reader to \cite{HPP15} for the details. 
 
For a fixed positive integer $N$, we construct a Riemann surface $\M_{(N)}$ by ``pasting together'' the Riemann surfaces $\M_\ell$ to which all functions $G_\ell$, with $\ell=0,1,\dots,N$, admit meromorphic extension. Namely, we set
\begin{equation}
\label{eq:MN}
\M_{(N)}
=\left\{(z,\zeta)\in \C^- \times \C^{N+1} ; \zeta=(\zeta_0,\dots, \zeta_N),\
(z,\zeta_\ell)\in \M_\ell, \ \ell\in\Zb_{\geq 0}, \ 0\leq \ell\leq N\right\}.
\end{equation}
Then $\M_{(N)}$ is a Riemann surface, and the map
\begin{eqnarray}\label{finite cover 1}
&&\pi_{(N)}: \M_{(N)} \ni(z,\zeta)\to z\in \C^-
\end{eqnarray}
is a holomorphic $2^{N+1}$-to-$1$ cover, except when
$z=-i L_\ell$ for some $\ell\in\Zb_{\geq 0}$ with $0\leq \ell\le N$.
The fiber above each of these elements $-i L_\ell$ consists of $2^N$ branching points of $\M_{(N)}$. A choice of square root function $\zeta_\ell^+(z)$, see \cite[(81)]{HPP15}, 
for every coordinate function $\zeta_\ell$ on $\M_{(N)}$ yields a section
\[
\sigma_{(N)}^+: z\to (z,\zeta_0^+(z),\dots,\zeta_N^+(z))
\]
of the projection $\pi_{(N)}$. All possible sections of $\pi_{(N)}$ are obtained by choosing a sign $\pm\zeta_\ell^+$ for each coordinate function. We obtain in this way a parametrization of all 
sections of $\pi_{(N)}$ by means of elements $\varepsilon=(\varepsilon_0,\dots,\varepsilon_N)\in \{\pm 1\}^{N+1}$.

For $0\leq \ell\leq N$ consider the holomorphic projection
\begin{eqnarray}\label{finite cover pi}
&&\pi_{(N,\ell)}: \M_{(N)} \ni(z,\zeta)\to (z,\zeta_\ell)\in \M_\ell.
\end{eqnarray}
Then the meromorphic function
\begin{equation}
\label{eq:wtGNell}
\wt G_{(N,\ell)}=\wt G_{\ell}\circ \pi_{(N,\ell)}: \M_{(N)}\to \C
\end{equation}
is holomorphic on ${(\pi_{(N)})}^{-1}(\C^-\setminus i\R)$. Moreover, on $\C^-\setminus i\R$,
\[
\wt G_{(N,\ell)}\circ  \sigma_{(N)}^+=G_{\ell}\,.
\]
So,  $\wt G_{(N,\ell)}$ is the meromorphic extension of a lift of $G_{\ell}$ to $\M_{(N)}$.
Using the right-hand side of \eqref{eq:Fresidues2} with $F_{(n)}$ constant on the $z$-fibers, we   obtain a lift of $F$ to $\pi_{(N)}^{-1}(W_{(n)}\setminus i\R)$.

The next step is to ``glue together'' all these local meromorphic extensions of $F$, moving from branching point to branching point, to get a meromorphic extension of $F$ along the branched curve $\gamma_N$ in $\M_{(N)}$ covering the interval $-i(0,L_{N+1})$.
Define, as in \cite[section 4.3]{HPP15}, the open sets $U_{n,\eps}$, $U_{\eps(n^\vee)}$  (with $n\in \Zb_{\geq 0}$, $\eps\in \{\pm 1\}^{N+1}$) and the open neighborhood $\M_{\gamma_N}$ of $\gamma_N$ in $\M_{(N)}$. Every open set $U_{\eps(n^\vee)}\cup U_{n,\eps}$ is a homeomorphic lift to $\M_{(N)}$ of the neighborhood $W_{(n)}$ of $-[L_n,L_{n+1})$.
Then we have the following analogue of \cite[Theorem 19]{HPP15}.

\begin{thm}
\label{thm:meroliftF}
For $n\in \{-1,0,\dots,N\}$, $\eps \in \{\pm 1\}^{N+1}$ and $(z,\zeta) \in U_{\eps(n^\vee)}\cup U_{n,\eps}$ define
\begin{equation}
\label{eq:widetildeF}
\wt F(z,\zeta)=
\displaystyle{F_{(n)}(z)+4 \sum_{\ell=0}^n\wt G_{(N,\ell)}(z,\zeta)+4
\sum_{\stackrel{n<\ell\leq N}{\text{with $\eps_\ell=-1$}}} \big[\wt G_{(N,\ell)}(z,\zeta)-
\wt G_{(N,\ell)}(z,-\zeta) \big]}\,,
 \end{equation}
where the first sum is equal to $0$ if $\ell=-1$ and the second sum is $0$ if $\eps_\ell=1$ for all $\ell>n$.
Then $\wt F$ is the meromorphic extension of a lift of $F$ to the open neighborhood
$\M_{\gamma_N}$ of the branched curve $\gamma_{N}$
lifting $-i(0,L_{N+1})$ in $\M_{(N)}$.
\end{thm}

Order the singularities according to their distance from the origin $0\in \C$, and let
$\{z_{(h)}\}_{h\in \Zb_{\geq 0}}$ be the resulting ordered sequence. For a fixed $h\in \Zb_{\geq 0}$  set
\begin{equation}
\label{eq:Sh}
 S_{h}=\{\ell\in \Zb_{\geq 0}; \ \text{$\exists k \in \Zb_{\geq 0}$ so that 
$b^{-2}|z_{(h)}|^2=(\rho_{\beta_1}+\ell)^2 +(\rho_{\beta_2}+\ell+k)^2$}\}\,.
\end{equation}
Notice that if $\ell\in S_h$, then the corresponding element $k$ is uniquely determined. 
Let $N\in \Zb_{\geq 0}$ be such that $|z_{(h)}|<L_{N+1}$ and $n\in \leftFpar 0, N\rightFpar$ such that $|z_{(h)}|\in [L_n,L_{n+1})$. 
Then the possible singularities of $\wt F$ at points of $\M_{(N)}$ above $z_{(h)}$ are those of 
\[
\sum_{\ell=0}^n \wt G_{(N,\ell)}(z,\zeta)=\sum_{\ell=0}^n \wt G_\ell(z,\zeta_\ell)\,.
\]
Indeed, the singularities of $\wt G_{(N,\ell)}(z,\zeta)=\wt G_\ell(z,\zeta_\ell)$ occur at 
points $(z,\zeta)\in \M_{(N)}$ with $|z|^2=L_\ell^2+L_m^2>L_\ell^2$. Hence the second sum on the right-hand side of \eqref{eq:widetildeF} is holomorphic on $U_{\eps(n^\vee)}\cup U_{n,\eps}$.

The singular points of $\wt F$ above $z_{(h)}$ are parametrized by $\varepsilon\in \{\pm 1\}^{N+1}$. We denote by $(z_{(h)}, \zeta^{(h,\varepsilon)})$ the one in $U_{\eps(n^\vee)}\cup U_{n,\eps}$. The local expression of $\wt F$ on $U_{\eps(n^\vee)}\cup U_{n,\eps}$ is computed in terms of the chart $\kappa_{n,\eps}$ defined for $(z,\zeta) \in U_{\eps(n^\vee)}\cup U_{n,\eps}$ by $\kappa_{n,\eps}(z,\zeta)=\zeta_n$. 

Suppose $\wt G_{(N,\ell)}(z,\zeta)$ is singular at $(z_{(h)}, \zeta^{(h,\varepsilon)})$. Then, by \cite[Proposition 21]{HPP15}, 
\begin{equation}
\label{eq:compres1}
\Res_{\zeta_n=\zeta^{(h,\eps)}_n} \big( \wt G_{(N,\ell)}
\circ \kappa_{n,\eps}^{-1}\big)(\zeta_n)
=
\eps_\ell\eps_n \;\frac{L_n^2}{L_\ell^2}  \; \frac{\sqrt{|z_{(h)}|^2-L_\ell^2}}{\sqrt{|z_{(k)}|^2-L_n^2}}\; \Res_{\zeta_\ell=\zeta^{(h,\eps)}_\ell} \big( \wt G_{\ell}
\circ \kappa_{\ell,-}^{-1}\big)(\zeta_\ell)\,.
\end{equation}
If $\ell$ satisfies \eqref{eq:singular1} with $z=z_{(h)}$ for some $k \in \Zb_{\geq 0}$, then 
$|z_{(h)}|^2-L_\ell^2=b^2(\rho_{\beta_2}+\ell+k)^2=L^2_{\ell+\frac{m_\mrm}{2}+k}$\,.
If $\ell\geq \frac{m_\mrm}{2}$ satisfies \eqref{eq:singular2} with $z=z_{(h)}$ for some $m \in 
\leftFpar 0,\ell_0\rightFpar$ and $\ell_0=\ell-\frac{m_\mrm}{2}$, then 
$|z_{(h)}|^2-L_\ell^2=b^2(\rho_{\beta_1}+m)=L^2_m$\,.

In the first case, by \eqref{eq:resGell1}, the right-hand side of \eqref{eq:compres1} is equal to
\begin{multline*}
\frac{\eps_n L_n^2}{\sqrt{|z_{(h)}|^2-L_n^2}}  \; \frac{L_{\ell+\frac{m_\mrm}{2}+k}}{L_\ell^2}\;  \Res_{\zeta_\ell=-\zeta_{\ell,\ell+\frac{m_\mrm}{2}+k}} \big( \wt G_{\ell}
\circ \kappa_{\ell,-}^{-1}\big)(\zeta_\ell)\\
=\frac{i}{\pi^2}\; \frac{\eps_n L_n^2}{\sqrt{|z_{(h)}|^2-L_n^2}}  \; \frac{L_{\ell+\frac{m_\mrm}{2}+k}}{L_\ell^2}\;  C_{\ell,\ell+\frac{m_\mrm}{2}+k} \big( f\times \varphi_{\lambda(\ell,\ell+k)}\big)(y)\,.
\end{multline*}
In the second case, by \eqref{eq:resGell2}, the right-hand side of \eqref{eq:compres1} is equal to
\[
\frac{\eps_n L_n^2}{\sqrt{|z_{(h)}|^2-L_n^2}}  \; \frac{L_{m}}{L_\ell^2}\;  \Res_{\zeta_\ell=-\zeta_{\ell,m}} \big( \wt G_{\ell}
\circ \kappa_{\ell,-}^{-1}\big)(\zeta_\ell)
=\frac{i}{\pi^2}\;\frac{\eps_n L_n^2}{\sqrt{|z_{(h)}|^2-L_n^2}}  \; \frac{L_{m}}{L_\ell^2}\;   C_{\ell,m} \big( f\times \varphi_{\lambda(m,\ell_0)}\big)(y)\,.
\]
Observe that in both cases, the constants appearing are $i$ times a positive constant. 
Observe also that if $\ell\geq \frac{m_\mrm}{2}$ and $\wt G_{(N,\ell)}$ is singular at 
$(z_{(h)}, \zeta^{(h,\varepsilon)})$ with $\ell$ satisfying \eqref{eq:singular2} with $z=z_{(h)}$, some $m \in \leftFpar 0,\ell_0\rightFpar$ and $\ell_0=\ell-\frac{m_\mrm}{2}$, then 
$(z_{(h)}, \zeta^{(h,\varepsilon)})$ is also a singularity of $\wt G_{(N,m)}$ and 
$m$ satisfies \eqref{eq:singular1} with $z=z_{(h)}$ and $k=\ell_0-m\in \Zb_{\geq 0}$.
Of course, $\varphi_{\lambda(m,\ell_0)}=\varphi_{\lambda(m,m+k)}$ in this case. It follows that 
the set $S_h$ is sufficient to parametrize the residues of $\wt F$ at $(z_{(h)}, \zeta^{(h,\varepsilon)})$.

It follows that 
\begin{equation}
\label{eq:compresF}
\Res_{\zeta_n=\zeta^{(h,\eps)}_n} \big( \wt F
\circ \kappa_{n,\eps}^{-1}\big)(\zeta_n)=
\frac{i\eps_n L_n^2}{\sqrt{|z_{(h)}|^2-L_n^2}} \sum_{\ell\in S_h} c_\ell \big(f \times \varphi_{\lambda(\ell,\ell+k)}\big)(y)\,,
\end{equation}
where $k\in \Zb_{\geq 0}$ is associated with $\ell$ as in the definition of $S_h$ and $c_\ell$ is a positive constant depending only on $\ell$.

By Proposition \ref{pro:holoextRF}, the meromorphic extensions on the half-line $i(-\infty, -L]$ of $F$ and of the resolvent $R$ of the Laplacian are equivalent. Thus 
the resolvent $R$ can be lifted and meromorphically extended along the curve $\gamma_{N}$ in $\M_{\gamma_N}$. Its singularities 
(i.e. the resonances of the Laplacian) are those  of the meromorphic extension $\wt F$ of $F$ and are located at the  points 
of $\M_{\gamma_N}$ above the elements $z_{(h)}$. They are simple poles. 
The precise description is given by the following theorem.

\begin{thm}
\label{thm:meroextshiftedLaplacian}
Let $f \in C^\infty_c(\X)$ and $y \in \X$ be fixed. Let $N\in \mathbb N$ and let $\gamma_{N}$ be the curve lifting the inteval $-i(0,N+1)$ in $\M_{(N)}$.
Then the resolvent $R(z)=[R(z)f](y)$ lifts as a meromorphic function to the neighborhood $\M_{\gamma_{N}}$ of the curve $\gamma_{N}$ in $\M_{(N)}$. We denote the lifted meromorphic function by
$\wt R_{(N)}(z,\zeta)=\big[\wt R_{(N)}(z,\zeta)f\big](y)$.

The singularities of $\wt R_{(N)}$ are at most simple poles at the points $(z_{(h)},\zeta^{(h,\eps)})\in \M_{(N)}$ with $h\in \Zb_{\geq 0}$ 
so that $|z_{(h)}|<L_{N+1}$ and $\eps \in\{\pm 1\}^{N+1}$. Explicitly, for $(n,\eps)\in \leftFpar 0,N\rightFpar\times \{\pm 1\}^{N+1}$,
\begin{eqnarray}
\label{eq:RNnearwn}
\wt R_{(N)}(z,\zeta)=\wt H_{(N,m,\eps)}(z,\zeta)+2\pi i \sum_{\ell=0}^m \wt G_{(N,\ell)}(z,\zeta) \qquad
((z,\zeta)\in  U_{\eps(n^\vee)} \cup U_{n,\eps})
\,,
\end{eqnarray}
where $\wt H_{(N,m,\eps)}$ is holomorphic and $\wt G_{(N,\ell)}(z,\zeta)$
is in fact independent of $N$ and $\eps$ (but dependent on $f$ and $y$, which are omitted from the notation).
The singularities of $\wt R_{(N)}(z,\zeta)$ in $U_{\eps(n^\vee)} \cup U_{n,\eps}$ are simple poles at the points $(z_{(h)},\zeta^{(h,\eps)})$ belonging to $U_{\eps(n^\vee)} \cup U_{n,\eps}$. The residue of the local expression of $\wt R_{(N)}$ at one such point is $i\pi$ times the right-hand side of \eqref{eq:compresF}. 
\end{thm}

\section{The residue operators}

Recall the notation $\lambda(\ell_1,\ell_2)= (\rho_{\beta_1}+\ell_1)\beta_1+(\rho_{\beta_2}+\ell_2)\beta_2$ introduced in \eqref{eq:lambdal1l2}.
For a fixed $h\in \Zb_{\geq 0}$, the sum over $S_h$ appearing on the right-hand side of \eqref{eq:compresF} is independent either of $N$ or $n$. It can be used to define 
the residue operator ${\Res}_{z_{(h)}}\wt R$ of the meromorphically extended resolvent at $z_{(h)}$. 
Explicitly, 
\begin{equation}\label{a1}
{\Res}_{z_{(h)}}\wt R=\sum_{\ell\in S_{h}} c_\ell R_{\lambda(\ell,\ell+k_\ell)}
\end{equation}
where, $c_\ell$ are non-zero constants and, as in \cite[(57)]{HPP14}, $R_{\lambda}: C_c^\infty(\X)\to C^\infty(\X)$ is defined by $R_{\lambda}f=f\times \varphi_{\lambda}$. 
We know from \cite[chapter IV, Theorem 4.5]{He2} that $R_{\lambda}(C_c^\infty(\X))$ is an irreducible representation of $\G$. Furthermore, two such representations are equivalent if and only if the spectral parameters $\lambda$ are in the same Weyl group orbit. Since, in our case, the Weyl group acts by transposition and sign changes, the element $\lambda(\ell_1,\ell_2)$
is dominant with respect to the fixed choice of positive roots if and only if 
$$
\rho_{\beta_2}+\ell_2\geq \rho_{\beta_1}+\ell_1\geq 0\,, \quad \text{i.e.}\quad 
\ell_2+ \frac{m_\mrm}{2}\geq \ell_1\,.
$$
In particular, all $\lambda(\ell,\ell+k_\ell)$ are distinct and dominant.
Hence, as a $\G$-module, 
\begin{equation}\label{a2}
{\Res}_{z_{(h)}}\wt R(C_c^\infty(\X))=\bigoplus_{\ell\in S_{(h)}} R_{\lambda(\ell,\ell+k_\ell)}(C_c^\infty(\X))\,.
\end{equation}
\begin{thm}\label{a3ell}
If $(\ell, k)\in \Zb_{\geq 0}^2$, then 
$
\dim R_{\lambda(\ell,\ell+k)}(C_c^\infty(\X))<\infty\,.
$
Thus ${\Res}_{z_{(h)}}\wt R(C_c^\infty(\X))$ is a finite dimensional $\G$-module.

The $\G$-module ${\Res}_{z_{(h)}}\wt R(C_c^\infty(\X))$ has $\K$-finite matrix coefficients (and is unitary)  if and only 
it is the trivial representation, which occurs for $h=0$, i.e. when $z_{(0)}=-i\sqrt{\inner{\rho}{\rho}}$.
\end{thm}
\begin{proof}
By \cite[Ch II, \S 4, Theorem  4.16]{He3}, $\dim R_{\lambda(\ell_1,\ell_2)}(C_c^\infty(\X))<\infty$
if and only if there is $w\in W$ such that 
\begin{equation}\label{a4}
(w\lambda(\ell_1,\ell_2)-\rho)_\beta\in \Bbb Z_{\geq 0} \qquad (\beta\in \Sigma_*^+, \text{$\beta$ simple})\,.
\end{equation}
Recall that the simple roots in $\Sigma_*^+$ are $\beta_1$ and $\frac{\beta_2-\beta_1}{2}$.
Moreover, $\mu_{\frac{\beta_2-\beta_1}{2}}=\mu_{\beta_2} -\mu_{\beta_1}$ for $\mu\in \a_\C^*$. 
Since 
$$
\lambda(\ell,\ell+k)-\rho=(\rho_{\beta_1}+\ell)\beta_1+(\rho_{\beta_2}+\ell+k)\beta_2-
(\rho_{\beta_1}\beta_1+\rho_{\beta_2}\beta_2)=\ell \beta_1+(\ell+k)\beta_2\,,
$$
we conclude that 
\begin{eqnarray*}
&&(\lambda(\ell,\ell+k)-\rho)_{\beta_1}=\ell\in \Zb_{\geq 0}\,,\\
&&(\lambda(\ell,\ell+k)-\rho)_{(\beta_2-\beta_1)/2}=k\in \Zb_{\geq 0}\,,
\end{eqnarray*}
which satisfies \eqref{a4} with $w=\id$.
\end{proof}


\end{document}